\documentclass[12pt,twoside]{amsart}
\usepackage[latin1]{inputenc}
\usepackage{times}
\usepackage{amsthm}
\usepackage{amssymb, verbatim}
\usepackage{amsmath}
\usepackage{eucal}
\usepackage{mathrsfs}

\topmargin = - 40 pt 
\usepackage[all]{xy}

\newcommand     {\Cset}    {{\mathbb C}}
\newcommand     {\Nset}    {{\mathbb N}}
\newcommand     {\Zset}    {{\mathbb Z}}
\newcommand     {\Qset}    {{\mathbb Q}}

\DeclareMathOperator{\op}{op}
\DeclareMathOperator{\Irr}{Irr}

\newcommand     {\A}    {{\mathcal A}}
\newcommand     {\B}    {{\mathcal B}}
\newcommand     {\C}    {{\mathcal C}}

\renewcommand   {\H}    {{\mathcal H}}
\newcommand     {\M}    {{\mathcal M}}
\newcommand     {\Q}    {{\mathcal Q}}
\renewcommand   {\S}    {{\mathcal S}}
\newcommand     {\T}    {{\mathcal T}}
\newcommand     {\U}    {{\mathcal U}}
\newcommand     {\F}    {{\mathcal F}}
\newcommand     {\OK}   {{\mathcal O}_K}

\newcommand {\go} {{G^{\circ}}}

\DeclareMathOperator    {\Rep}  {Rep}
\DeclareMathOperator    {\Hilb}  {Hilb}
\DeclareMathOperator    {\SU}  {SU}
\DeclareMathOperator    {\SO}  {SO}
\DeclareMathOperator    {\rk}  {rk}

\DeclareMathOperator{\Cyc}{C}
\DeclareMathOperator{\Dih}{D}

\newtheorem{thm}{Theorem}[section]

\newtheorem{lemma}[thm]{Lemma}
\newtheorem{prop}[thm]{Proposition}
\newtheorem{cor}[thm]{Corollary}

\theoremstyle{definition}

\newtheorem{defn}[thm]{Definition}

\newtheorem{ex}[thm]{Example}
\newtheorem{exs}[thm]{Examples}

\newtheorem{rem}[thm]{Remark}

\numberwithin{equation}{section}

\begin{document}

 \title[Connected components of cmqg]
{Connected components of compact matrix quantum groups\\   and  finiteness  conditions}
\author[L.~Cirio]{Lucio S. Cirio}
\email{lucio.cirio@uni.lu}
\address{
Unit\'e de Recherche en Math\'ematiques, Universit\'e du Luxembourg\\
6, rue Richard Coudenhove-Kalergi, L--1359 Luxembourg, Luxembourg}
\author[A.~D'Andrea]{Alessandro D'Andrea}
 \author[C.~Pinzari]{Claudia Pinzari}
\author[S.~Rossi]{Stefano Rossi}
\email{dandrea@mat.uniroma1.it\\
pinzari@mat.uniroma1.it\\
stipan@alice.it}
\address{Dipartimento di Matematica, Universit\`a degli Studi di
Roma ``La Sapienza''\\ P.le Aldo Moro, 5 -- 00185 Rome, Italy}

\begin{abstract} We introduce the notion of identity component  of a compact quantum group and that of total disconnectedness. As a drawback of the generalized Burnside problem, we note that totally disconnected compact matrix quantum groups  may fail to be profinite.  We consider the problem of approximating the identity component as well as the maximal normal (in the sense of Wang) connected subgroup by introducing
 canonical, but possibly transfinite, sequences of subgroups. These sequences have a trivial behaviour
in the classical case. We give examples, arising as free products, where the identity component is not normal and the associated sequence has length $1$.

We give necessary and sufficient conditions for normality of the identity component and finiteness or profiniteness of the quantum component group. Among them, we introduce an ascending chain condition on the representation ring, called Lie property, which characterizes Lie groups in the commutative case and reduces to group Noetherianity of the dual in the cocommutative case. It is weaker than ring Noetherianity but ensures existence of a generating representation. The Lie property and ring Noetherianity are inherited by quotient quantum groups. We show that $A_u(F)$ is not of Lie type. We discuss an example arising from the compact real form of $U_q(\mathfrak{sl}_2)$ for $q<0$.
\end{abstract}


\maketitle

\tableofcontents

\section{Introduction}

The problem of extending L.~S.~Pontryagin's duality theory for  locally compact commutative groups to the non-commutative case was one of the first motivations for the development of the theory of Hopf algebras, or quantum groups, with the work  of G.~I.~Kac \cite{Kac} in the 1960s, see   \cite{ES}. Since then, Hopf algebras have been extensively studied both in a purely algebraic and an operator algebraic setting.
V.~G.~Drinfeld and M.~Jimbo obtained natural important examples as deformation of classical groups in the mid 80s \cite{Drinfeld1, Drinfeld2, Jimbo}, in the framework of Lie theory. They were originally motivated by the theory of integrable quantum systems.

On the other hand, the theory of $C^*$-algebras originated in the 1940s with the need of a mathematical framework for quantum mechanics.
A $C^*$-algebra is usually interpreted as a non-commutative topological space. This interpretation is stressed by the Gelfand transform, which gives a natural duality between locally compact Hausdorff spaces and commutative $C^*$-algebras. Noncommutative geometry, in the approach of A.~Connes, aims at furnishing the non-commutative space underlying a general $C^*$-algebra with geometric structures \cite{Connes}.

S.~L.~Woronowicz initiated an investigation of quantum groups motivated by noncommutative geometry  \cite{Wsu2}. In his approach a quantum group is described by a Hopf $C^*$-algebra. He gave an axiomatization of compact quantum groups admitting a generating finite-dimensional representation, the {\em compact matrix quantum groups}. He established existence of the Haar measure, as well as Peter-Weyl theory \cite{Wcmp, Wtk}. He later removed the finite-generation assumption, and formulated a general representation  theory for {\em compact quantum groups} \cite{Wcqg}.

Compact matrix quantum groups generalize not only compact Lie groups, but also finitely generated discrete groups. Indeed, the latter arise as dual objects, when all irreducible representations of the quantum group are one-dimensional (cocommutative examples). A general compact matrix quantum group, being an intermediate object between these special cases, exhibits aspects of both.

Woronowicz obtained important examples by deforming the algebra of continuous functions on the special unitary groups \cite{Wsu2, Wtk}. These were shown to be closely related to the examples of Drinfeld and Jimbo  via Tannaka--Krein duality \cite{Rosso}. New examples, the groups $A_u(F)$ and $A_o(F)$, were introduced by S.~Wang and A.~Van Daele \cite{Wang_free, WVD}. They are not deformations of classical groups, and are highly noncommutative in a suitable sense (apart from the special case $A_o(F), \rk F=2$, which is isomorphic to some ${\rm SU}_q(2)$ \cite{Banica}). They are often referred to as the {\em free unitary} and {\em orthogonal} quantum groups, respectively. Their representation theory has been studied by T.~Banica \cite{Banica}. A wealth of new examples has been described over time, see \cite{ Banica5, Banica3, Banica4, Wang2} and references therein.


The generalization of topological and differential-geometric notions to the compact quantum group setting is a challenging problem, comprising a reformulation of the geometry of compact Lie groups in global  terms, as explained by M.~A.~Rieffel in \cite{Rieffel}.
Notably, S. Wang proposed the problem of developing  Cartan-Weyl theory for connected compact matrix quantum groups \cite{Wang_problems}, he introduced the notion of simple compact quantum group, analysed its validity in the free as well as deformed examples and suggested the problem
of investigating the structure of compact quantum groups in terms of simple ones \cite{Wang}; see also \cite{Chirvasitu} for almost simplicity of the unitary free quantum groups.

One of the main difficulties of Wang's project relies in establishing to what extent noncommutative generalizations of classical concepts will play as prominent a role as in the classical case to allow a development of analogous theories.
Indeed, the variety of compact matrix quantum groups is so wide that not much general structure is known, or is likely to emerge, as compared to the theory of compact Lie groups. This can already be observed among the cocommutative examples, the trouble being that the full complexity of finitely generated discrete  groups appears as their duals.

As a very simple example, while compact matrix quantum groups are closed under passage to quantum subgroups, they are not so under formation of quotients, since, in the cocommutative case, these correspond via duality to subgroups of finitely generated groups, which are not finitely generated in general.

We intend to introduce additional constraints that restrict the class of compact matrix quantum groups to a subclass which can hopefully be treated along the lines of the theory of compact Lie groups and also benefit from ideas of geometric group theory. We are interested in selecting this subclass in agreement with the general paradigm of Connes' approach, according to which the richness of the geometric structure depends on the amount of noncommutativity that one decides to afford. Following M.~A.~Rieffel, noncommutative geometries are regarded as approximating classical geometries \cite{Rieffel2}. In this sense, the class should exclude highly non-commutative examples, and certainly include such examples as those arising from deformation of the classical groups. These typically have representation rings isomorphic to that of their classical limit, hence commutative.
More generally, we aim to limit the amount of noncommutativity in the representation theory.



An ubiquitous topological aspect of Lie groups is connectedness. Compact Lie groups are almost connected, in the sense that they have finitely many connected components. Almost connectedness is a simple but important structural property, in that it allows to reduce their study to the connected case, and eventually leads to complete classification.
In this paper, we aim to identify a class of compact matrix quantum groups for which an analogue of almost connectedness holds.

We believe that this is a first step towards a long-term project that focuses on constructing geometric structures, where a further refinement will likely be needed. Indeed, as indicated by
the well known solution  to Hilbert's fifth problem by A.~Gleason and D.~Montgomery and L.~Zippin
\cite{Gleason, MZ}, local connectedness, together with finite dimensionality, gives a characterization of Lie groups among locally compact groups, see also \cite{Tao} and references therein.

It turns out, as we shall show,  that for general compact {\em matrix} quantum groups, disconnectedness
is an extremely intricate and rich matter of investigation. For example, it includes the Burnside problem in discrete group theory, and this fact prevents  almost connectedness in the general case. Moreover, it exhibits unexpected phenomena as compared not only to the classical but also to the  cocommutative case.
For example, the identity component may fail to be normal, in the sense of S.~Wang \cite{Wang_free}. However, these examples have highly noncommutative representation theories. This indicates that the behaviour of the notion of connectedness we refer to, depends on the level of noncommutativity involved.

We shall indeed show that a class of almost connected compact matrix quantum groups with normal identity component can be precisely detected among those quantum groups  for which the torsion part of the  representation theory is commutative and normal, in a suitable sense, but more is needed. The crucial  property  emerging from our analysis is a finite chain condition on the representation ring of the quantum group, that we call Lie property. It imposes a finiteness condition on increasing chains of quotient quantum groups. This result is precisely stated  in Theorem \ref{final} and will be described in more detail later. For example, the Lie property holds if the the representation ring is isomorphic to that of a compact Lie group, and in this sense is natural for our program.

For the quantum $\SU(2)$ group, connectedness has been established by S.~L.~Woronowicz by means of differential calculus \cite{Wsu2}. In \cite{Wang} S.~Wang rephrased the notion of connectedness  for compact quantum groups in representation theoretic terms. His notion in fact goes back to L.~S.~Pontryagin's characterization of connected locally compact abelian groups via duality theory \cite{P}. A compact quantum group is called connected if the coefficients of every non-trivial representation generate an infinite-dimensional Hopf $^*$-subalgebra. Equivalently, every representation generates an infinite tensor subcategory with conjugates. Being formulated in representation theoretic terms, it relies on the group property in a fundamental way. 

Let $G$ be a compact quantum group. We start by observing that the set of connected quantum subgroups of $G$ is closed under the operation of taking the quantum subgroup generated by an arbitrary family, so it contains a unique maximal element $\go$: the {\em identity component} of $G$. Clearly, $G$ is connected if and only if $G=\go$. If $\go$ is the trivial group, we shall say that $G$ is {\em totally disconnected}.

Obviously, $\go$ reduces to the connected component of the identity if $G$ is a topological group, while, if $G$ is the dual of a discrete group $\Gamma$, it reduces to (the dual of) the universal torsion-free image  $\Gamma_f$ of $\Gamma$, as considered by S.~D.~Brodsky and J.~Howie \cite{BH}. In  that paper the authors give conditions implying that $\Gamma_f$ is locally indicable, i.e., every non-trivial finitely generated subgroup admits an epimorphism to ${\mathbb Z}$. The representation ring of a locally indicable group is an integral domain \cite{Higman}. Although beyond the specific aims of this paper, we find it quite remarkable that, when interpreted in the framework of quantum groups, the locally indicable groups correspond precisely to the cocommutative compact quantum groups $G$ admitting a $1$-dimensional classical torus as a quantum subgroup of every matrix quotient of $G$.

We consider the following problems, of a rather different nature: understanding  normality of the identity component, and finiteness of the non-commutative analogue of the component group $G^\circ\backslash G$ of a compact Lie group. The latter problem splits into two problems, reducing to the totally disconnected case: deciding whether $\go\backslash G$ is totally disconnected and under what conditions it is still a {\em matrix} quantum group, hence in turn involving the more fundamental problem of finite generation of quotients.

Thus a special case of our problem is that of whether a totally disconnected compact matrix quantum group is finite, and this has a negative answer, in general. One immediately realizes that this includes, for cocommutative quantum groups, the generalized Burnside problem, as mentioned before, i.e., deciding whether a finitely generated torsion group must be finite.
Indeed, if every irreducible representation of $G$ generates a finite tensor subcategory with conjugates (i.e., it is a {\em torsion representation}) then $G$ is totally disconnected, by Proposition \ref{sufficientfortotdisc}. In particular, cocommutative quantum groups corresponding to torsion groups are totally disconnected.

The Burnside problem was answered in the negative by E.~S.~Golod and I.~R.~Shafarevich for unbounded exponents \cite{GS, G} and by S.~I.~Adian and P.~S.~Novikov in the bounded case \cite{AN}. Such examples show, by Proposition \ref{Burnside}, that totally disconnected compact matrix quantum groups are not even {\em profinite} (cf. Definition \ref{profinite}). Hence the class of quantum groups whose all irreducible representations are torsion, contains the class of profinite quantum groups as a proper subclass. We shall later see that in fact it does not even exhaust the totally disconnected compact quantum groups.

Among classes of finitely generated torsion groups which are  known to be finite are the abelian groups, or, more generally, groups with finite conjugacy classes or the nilpotent ones.  Therefore, in the special case of totally disconnected quantum groups, the first class to consider for the finiteness problem is that for which the tensor product of two representations is commutative up to equivalence, see also Remark \ref{nonamenable}.

  We next describe our main results in more detail. The first one concerns normality of $G^\circ$: in Section \ref{normsect} we list many necessary and sufficient conditions. While $G^\circ$ is always normal in the commutative and cocommutative cases, we provide a class of examples arising as free products of compact quantum groups where $G^\circ$ is not normal.
Our result involves the following aspects. We start with a categorical characterization (Theorem \ref{carattnormal}) of quotient quantum groups by normal subgroups, which starts from  the results of \cite{PR};  we refer to the corresponding subcategories as {\em normal}. In the classical case tensor subcategories with conjugates are always normal, while in the cocommutative case, normal subcategories correspond to normal subgroups of $\Gamma$, if $G=C^*(\Gamma)$. We next study the {\em torsion subcategory} of the representation category of $G$ and its relation with the quotient space $\go\backslash G$. In the case of Lie groups, it coincides with the representation category of the latter, but already for cocommutative quantum groups it may be strictly smaller: for example, it may lack tensor products. The counterexamples to the Burnside problem show that direct sums of torsion objects may fail to be torsion,   even if the torsion subcategory has tensor products.

However, adjoining tensor products and direct sums may not suffice if the torsion subset is not a group. This is due to the fact that the quotient of a group by the subgroup generated by the torsion subset may still contain non-trivial torsion elements. An example has been constructed by M.~Chiodo in \cite{Chiodo}. From the quantum group viewpoint, this quotient is the dual of a subgroup, which is still disconnected. However, a simple ordinary inductive procedure \cite{Chiodo} yields
a sequence of quotients converging to the universal torsion-free quotient of \cite{BH},
whose length may be infinite or   arbitrarily  finite   (see Example \ref{arbitrary}).

If $G$ is a compact quantum group, we consider the unique maximal normal connected subgroup $G^{\rm n}$ of $G$. Clearly, there is an inclusion $G^{\rm n}\subset G^\circ$, which becomes an equality precisely when $G^\circ$ is normal. We may consider the normal quantum subgroup $G_1$ of $G$ defined by the requirement that $\text{Rep}(G_1\backslash G)$ is the smallest normal tensor subcategory of $\text{Rep}(G)$ containing all torsion representations. We construct    a canonical, but possibly transfinite,  normal decreasing sequence, $G_0=G\supset G_1\supset\dots\supset G_\alpha\supset\dots$ of quantum subgroups of $G$.   Cardinality considerations show  that this sequence must stabilize.   We introduce the {\em (normal) torsion degree} of $G$ as the smallest ordinal $\delta$ such that $G_\delta=G_{\delta+1}$.
In the cocommutative case, $\delta\leq\omega$, the first countable ordinal, and the sequence  reduces to Chiodo's construction.
 If $G^{\rm n}\backslash G$ is finite, or more generally if all irreducible representations of $G^{\rm n}\backslash G$ are torsion, then $G$ has torsion degree $\leq1$.  This includes the familiar case of compact groups, but also the intricate examples arising in connection with the Burnside problem.

We show that the torsion degree of $G$ coincides with the smallest ordinal $\delta$ such that $G_\delta$ is connected, and $G_\delta=G^{\rm n}$ (Theorem \ref{ntd}).
We derive a characterization of normality of $G^\circ$ in terms of the sequence $G_\alpha$ (Corollary \ref{necsuff}). This characterization is useful to exhibit a large class of free product quantum groups $G$ with non-normal identity component and yet of torsion degree $1$. In particular,    the identity component $G^\circ$  can be  any given (compact) adjoint semi-simple Lie group, while $G^{\rm n}$ is the trivial group  (Theorem \ref{non-normal} and Corollary \ref{corolnon-normal}). In these cases, the tensor subcategory ${\mathcal T}_1$ generated by the torsion subcategory is not normal, a novelty if compared to the case of discrete groups. Moreover, examples where $G^\circ$ is not normal and ${\mathcal T}_1$ is normal but infinite, i.e., with infinitely many irreducible objects, can be constructed as well.

A generalization of Chiodo's method provides cocommutative examples of each torsion degree $\leq\omega$, the first countable ordinal. It is an interesting problem that of deciding whether the torsion degree can assume values $>\omega$, especially in the case of compact quantum groups with normal identity component.
A slight variation of our construction yields a second, subnormal decreasing transfinite sequence, which converges to $G^\circ$ under certain conditions (Theorems \ref{suffnormality} and \ref{variation}).

A necessary condition for normality of $G^\circ$ and finiteness of $G^\circ\backslash G$ is that $\Rep(G^\circ\backslash G)$ equals  the torsion subcategory $\Rep(G)^t$, and therefore that $\Rep(G)^t$ is tensorial, finite and normal.
Theorem \ref{connectedcomponent} shows that these last conditions are also sufficient,  and moreover  $G$ has torsion degree $\leq 1$.  Our proof  uses the bimodule construction and induction theory for tensor $C^*$-categories developed in \cite{PRinduction}.  In Corollary \ref{corollary} we derive normality of $\go$ for compact quantum groups whose associated dense Hopf $^*$-algebra is an inductive limit of Hopf $^*$-subalgebras of quantum groups of the previous kind. Examples  clarify that the assumptions of Theorem \ref{connectedcomponent} are independently needed   for normality of $G^\circ$ and torsion degree $\leq1$.

From the technical viewpoint,   one may reasonably  argue whether the theory of induction of \cite{PR} used in the proof of Theorem   \ref{connectedcomponent}  may be replaced by a more direct
argument. We note that an  analogue of  Clifford-Mackey theory for compact quantum groups would suffice. However, to our knowledge, this theory is not  available. We hope to develop it elsewhere \cite{DPP}.

While we give simple criterions for normality of a tensor subcategory (Proposition \ref{sufficientfornormality}), our result reduces the problem of normality of the identity component to that of finiteness and tensoriality of the torsion subcategory. As noted before, the first case to consider is that where tensor product of torsion representations is commutative. In this case, the torsion subcategory is tensorial. We may thus regard this problem as a special case of understanding whether a full tensor subcategory (with conjugates, subobjects and direct sums) of a finitely generated tensor category is still finitely generated, which, as mentioned above, is of interest in its own right, whether or not commutativity of the torsion part is assumed.

Indeed, full tensor subcategories of $\Rep(G)$ are in one-to-one correspondence with quotient quantum groups of $G$ and also with subhypergroups of the dual object $\hat{G}$. We shall refer to the subring of the representation ring $R(G)$ generated by a subhypergroup as a  {\em representation subring}. Hence the problem is one of finite generation of subhypergroups, which we frame as finite generation of representation subrings.
 From the geometric viewpoint, it becomes the problem of identifying a class of compact matrix quantum groups that is stable under taking quotients.

This leads to Section 6, which also contains main results. We introduce an ascending chain condition on representation subrings of $R(G)$. We refer to $G$, or  $R(G)$, as being of {\em Lie type}. The terminology is motivated by the classical case: if $G$ is a compact group, every quotient group arises from a normal closed subgroup; the Lie property thus becomes equivalent to the requirement that every decreasing sequence of normal closed subgroups of $G$ stabilizes, which is indeed one of the characterizations of compact Lie groups among compact groups.

By Theorem \ref{generatingrep}, compact quantum groups of Lie type are necessarily compact matrix quantum groups. However, not every compact matrix quantum group is of Lie type. Indeed, in the cocommutative case, being of Lie type translates into the ascending chain condition on subgroups, or, equivalently, to the property that every subgroup is finitely generated: such groups are called Noetherian. For example, the free group on two generators is not Noetherian. We show that an analogous result holds for compact quantum groups: $A_u(F)$ is not of Lie type (Theorem \ref{$A_u(F)$}). The Lie property is obviously inherited by representation subrings. It follows that the family of compact quantum groups of Lie type is closed under taking quotients. In particular, every quotient is still a compact matrix quantum group. Equivalently, every full tensor subcategory is finitely generated.

One may also require that the representation ring of a compact quantum group $G$ is Noetherian; then $G$ is automatically of Lie type. More precisely, Theorem \ref{lietype} provides a natural connection between quotient quantum groups of $G$ and certain ideals of its representation ring established by the integer dimension function. A natural class of examples are the compact quantum groups with commutative and finitely
generated  representation ring. Indeed, being Noetherian, they are of Lie type. We notice, however, that a subgroup of a quantum group with Noetherian representation ring is not necessarily of Lie type (Remark \ref{nosubgroups}).

In general, Noetherianity of the representation ring is strictly stronger than the Lie property (although it is equivalent in the classical case) as follows by an example due to S.~V.~Ivanov \cite{Ivanov} in the context of discrete groups. We recall for completeness that almost polycyclic groups have Noetherian group ring and that Ivanov's example was motivated by Olshanskii's example earlier mentioned, which is also the first known example of a  Noetherian group not almost polycyclic \cite{Olshanskii}.  Whether almost polycyclic groups are the only ones with Noetherian group ring (over a field) is a long-standing open problem.

The main application of the Lie property  is to the study of the torsion subcategory $\Rep(G)^t$ of a compact quantum group $G$. Namely, if $G$ is of Lie type and if $\Rep(G)^t$  is commutative then it is automatically tensorial and, more importantly, being finitely generated, it is finite. Combining with Theorem \ref{connectedcomponent} shows that if $\Rep(G)^t$ is in addition normal, then $G^\circ$ is normal and $G$ is almost connected (Corollary \ref{commthentensfinite} and Theorem \ref{final}).

We would like to mention a striking, closely related result of M.~Hashimoto, who showed, with methods of algebraic geometry, that any pure subalgebra of a commutative finitely generated algebra over a Noetherian ring is finitely generated \cite{hashimoto}. Indeed, we note that a representation subring is a direct summand subalgebra, and is therefore pure.

We conclude the paper with an example arising from the compact real forms of Drinfeld--Jimbo quantization of
$\mathfrak{sl}_2$, for real values of the deformation parameter, namely $U_q(\mathfrak{su}_{2})$ for $q>0$ and $U_q(\mathfrak{su}_{1,1})$ for $q<0$. With the methods developed in this paper, we explicitly compute the identity component of the dual compact quantum groups, recognise that it is normal and compute the quantum component group.
While the case $q>0$ is widely known to split as   ${\rm SU}_q(2)\times\Cyc_2$, we shall mostly focus on the case $q<0$. This example does not arise as a product of the identity component and the component group.
We would like to express our gratitude to  K. De Commer \cite{DeCommer} who kindly indicated to us the papers by  L.~I.~Korogodski and E.~Koelink and J.~Kustermans
\cite{Koelink_Kustermans, Korogodski} on the quantum $\widetilde{{\rm SU}}(1,1)$ group, and a possible connection with $\widehat{U_q(\mathfrak{su}_{1,1})}$ for $q<0$.

In the Appendix, we introduce the notion of image of a quantum subgroup of a compact quantum group in a quotient, and we observe that in general not every quantum subgroup of a quotient is of this form. We discuss necessary and sufficient conditions in some special cases. We believe that the results of this section are helpful  to clarify the novelties that emerge from quantum groups as compared to the classical theory in relation with the general problem
of total disconnectedness of the quantum component group.

The paper is organized as follows.   Section 2   establishes notation and recalls results that we shall need. In Section 3 we give a categorical characterization of  quotient quantum groups of a given compact quantum group that arise from quantum subgroups. We refer to the associated categories as being normal, and we establish the main properties. Section 4 is dedicated to the introduction of the identity component, the maximal connected normal subgroup and to totally disconnected compact quantum groups. In Section 5 we discuss the problem of normality and that of finiteness, profiniteness or total disconnectedness  of the quantum component group. We introduce the above mentioned transfinite sequences approximating $G^\circ$ and $G^{\rm n}$, and we construct  examples where $G^\circ$ is not normal. In Section 6 we introduce the Lie property of a compact quantum group and we compare it with Noetherianity of the representation ring and finite generation of the hypergroup. In the last part of this Section we draw conclusions from the main results of the paper. Finally, as already mentioned, Section 7 is dedicated to an example.

\medskip

\section{Preliminaries}

 \subsection{Compact quantum groups}\label{prel}\
\medskip

\noindent We fix the notation and recall some results about compact quantum groups, duality, subgroups, normal subgroups and quotient spaces.
\begin{defn} (\cite{Wcqg})  A {\em compact quantum group} $G = (Q, \Delta)$ is a unital $C^*$-algebra $Q$ together with a coassociative unital $^*$-homomorphism $\Delta: Q\to Q\otimes Q$, called {\em comultiplication}, to the minimal $C^*$-algebraic tensor product such that $(Q\otimes{\mathbb C})\cdot \Delta(Q)$ and $({\mathbb C}\otimes Q)\cdot \Delta(Q)$ are dense.
\end{defn}
Let $H$ be a finite-dimensional Hilbert space, and denote by $\B(H)$ the algebra of linear operators. A representation of a compact quantum group $G = (Q, \Delta)$ on $H$ is a unitary element $u$ of $\B(H)\otimes Q$ such that the comultiplication on matrix coefficients
$$ u_{\psi,\phi}:=\psi^*\otimes 1 \circ u \circ \phi\otimes 1, \qquad \phi, \psi \in H,$$
is given by
$$\Delta(u_{\psi,\phi})=\sum_k u_{\psi, e_k}\otimes u_{e_k,\phi},$$
where $(e_k)$ is an orthonormal basis. The matrix coefficients   $u_{e_r, e_s}$ associated to a fixed orthonormal basis will be simply denoted by  $u_{rs}$.

A remarkable and well known theorem states that the linear space $\Q$ of coefficients of representations of $G$ is a  canonical dense Hopf $^*$-subalgebra of $Q$ in the algebraic sense, i.e., it is equipped with antipode and counit, and the comultiplication takes values in the algebraic tensor product
$$\Delta: {\Q}\to {\Q}\odot{\Q}.$$
Most importantly, $G$ admits a unique Haar measure, i.e., a translation invariant state $h$ on $Q$, which is faithful on $\Q$.  In particular, the given norm on $\Q$ is bounded below by the norm defined by the Haar measure (reduced norm) and above by the maximal $C^*$-norm, which is finite  \cite{Wcqg}.

The reduced and maximal norm differ in general. We may complete $\Q$ in the reduced or maximal $C^*$-norm and obtain a compact quantum group, $G_{\text{red}}$ or $G_{\text{max}}$ respectively, having the same representations as $G$. If the maximal and reduced norm coincide, $G$ is called coamenable. As the term indicates, this is an amenability property of the   representation theory of $G$.
 For regular multiplicative unitaries, coamenability has been introduced by Baaj and Skandalis \cite{BS}, and for compact quantum groups by Banica \cite{Banica2}. See also \cite{BMT}.

In this paper, an important role is played by the {\em cocommutative} examples, defined as follows. Let  $\Gamma$ be a discrete group. The group $C^*$-algebra, $C^*(\Gamma)$, is a compact quantum group $G$ with the usual comultiplication extending $\gamma\mapsto\gamma\otimes\gamma$, $\gamma\in \Gamma$. Irreducible representations are one-dimensional with coefficients given by the elements of $\Gamma$, hence $\Q=\Cset\Gamma$.
The Haar measure is given by evaluation at the identity and is a trace. $G$ is coamenable  if and only if $\Gamma$ is an amenable group.

On the other hand, if all the irreducible representations of a compact quantum group $G$ are one-dimensional, then they form a group, say $\Gamma$, under tensor products and conjugation.
Therefore the associated Hopf $C^*$--algebra $Q_G$ contains the group algebra $\Cset\Gamma$ as its canonical dense Hopf $^*$-subalgebra. Hence, up to the choice of the norm completion, $Q_G$ is of the form $C^*(\Gamma)$.
In this sense, cocommutative examples should be regarded as the non-commutative analogue of compact abelian groups. Indeed, Woronowicz calls such quantum groups {\em abelian} \cite{Wcmp}.

Further well known examples of compact quantum groups  are ${\rm SU}_q(d)$ \cite{Wsu2, Wtk}, which is coamenable \cite{Nagy}, $A_o(F)$ \cite{Wang_free, WVD}, coamenable if and only if $F$ has rank $2$ (a result ascribed to  Skandalis in \cite{Banica}), while $A_u(F)$ \cite{Wang_free, WVD} is never coamenable \cite{Banica}.

\subsection{Tannaka--Krein--Woronowicz  duality}\label{tannaka}\
\medskip

If $G$ is a compact quantum group, let $\Rep(G)$ be the category whose objects are finite-dimensional representations of $G$ and whose arrows are defined by
$$(u,v):=\{T\in\B(H_u, H_v): T\otimes 1\circ u=v\circ T\otimes 1\}.$$
This category has a natural structure of tensor $C^*$-category with conjugates, subobjects and direct sums in the sense of \cite{LR}. A  conjugate representation of $u$ will be denoted by $\overline{u}$ and the tensor product of objects by $uv$ and of arrows by $S\otimes T$. The trivial representation is the tensor unit and will be denoted by $\iota$. Every finite-dimensional representation is the direct sum of irreducible representations, hence the category is semisimple. If $u$ is a representation, a conjugate representation $\overline{u}$ is characterized by the existence of intertwiners $R\in(\iota, \overline{u}u)$, $\overline{R}\in(\iota, u\overline{u})$ solving the conjugate equations in the sense of \cite{LR}. It follows that $\overline{u}$ is unique up to unitary equivalence, $u$ is a conjugate of $\overline{u}$, both $\overline{u}u$ and $u\overline{u}$ contain the trivial representation, and two-sided Frobenius reciprocity holds, in the sense that there are natural linear isomorphisms
\begin{equation}\label{frobenius}
(v, wu)\simeq(v\overline{u}, w), \qquad (v, uw)\simeq(\overline{u}v, w).
\end{equation}
In particular, if $u$ is irreducible, $\overline{u}$ is irreducible as well and the  the spaces of arrows $(\iota, \overline{u}u)$, $(\iota, u\overline{u})$  have dimension $1$.
\begin{ex}
If $G$ arises from a discrete group $\Gamma$, tensor product and conjugate in $\Rep(G)$ correspond respectively to multiplication and inverse in $\Gamma$.
\end{ex}

Often,  compact quantum groups  are described via their representation category. The algebraic and the categorical approach are explicitly linked by a version of the Tannaka-Krein duality developed by Woronowicz \cite{Wtk}. Since this dual viewpoint will play a role in our paper, we briefly recall the necessary formalism.

When considered as an abstract category, $\Rep(G)$ does not determine $G$. For example, the representation categories of ${\rm SU}_q(2)$ and $A_o(F)$
are isomorphic \cite{Banica} as abstract tensor $C^*$-categories for a large class of choices for the matrix $F$. In order to recover $G$, we need to take into account
the embedding functor into the category of Hilbert spaces,
$$H: \Rep(G)\to\Hilb,$$
associating with any representation $u$ its Hilbert space $H_u$ and acting trivially on arrows.
Tannaka's duality is the process of recovering the dense Hopf algebra $(\Q, \Delta)$ from   $(\Rep(G), H)$.
Indeed, $\Q$ is linearly isomorphic, as a linear space, to the  algebraic direct sum of $\overline{H}_\alpha\otimes H_\alpha$, where $\alpha$ labels a complete set of irreducible representations and $H_\alpha$ is the Hilbert space of $\alpha$. The Hopf $^*$-algebra structure of $\Q$ is explicitly determined \cite{Wtk} by the fusion and conjugation structure of ($\Rep(G)$, $H$).

Abstract
tensor $C^*$-categories  do not generally embed
into the Hilbert spaces. In fact, those which do embed, after completion with subobjects and direct sums, are precisely the representation categories of the compact quantum groups.
 More precisely, for any given tensor $C^*$-category  $\T$ with conjugates, subobjects and direct sums and an embedding functor $\F:\T\to\Hilb$, there exists a compact quantum group $G$ such that ($\T$, $\F$) is isomorphic to $(\Rep(G),H)$.

\medskip

\subsection{Quantum subgroups and their quotient spaces}\
\medskip

\noindent  The notion of quantum subgroup for compact quantum groups is due to Podles \cite{Podles}. Since in this paper we adopt an algebraic approach, it will be convenient to consider a slight variation, see also \cite{Pembedding}, which identifies quantum subgroups with the same representation category. The two notions coincide if we focus on coamenable subgroups.
More precisely, recall that Podles defined a quantum subgroup  of $G=(Q_G, \Delta_G)$ to be a compact quantum group $K=({Q}_K, \Delta_K)$ together with a $^*$-epimorphism $\pi: { Q}_G\to{ Q}_K$ satisfying $\Delta_K\circ\pi=\pi\otimes\pi\circ\Delta_G$. Here, $\pi$ should be thought of as the analogue of the restriction map.

While quantum groups correspond to embedded tensor categories, in the Tannakian formalism, quantum subgroups correspond to inclusions of embedded categories. For any representation $u\in\Rep(G)$, we set $u\upharpoonright_K:= 1\otimes\pi\circ u$; this is a representation of $K$, referred to as the {\em restricted representation}.  Hence $\pi$ takes $\Q_G$ into $\Q_K$, and actually $\pi(\Q_G)=\Q_K$ since $\pi(\Q_G)$ is dense.

In this paper, a compact quantum group $K$ will be called a subgroup if there is an epimorphism between the dense Hopf $^*$-algebras  $\pi:\Q_G\to\Q_K$ compatible with comultiplications. Any irreducible representation of $K$ is a subrepresentation of some restricted representation. The map $u\in\Rep(G)\to u\upharpoonright_K\in\Rep(K)$ is a tensor $^*$-functor compatible with the embeddings of these representation categories into $\Hilb$. Conversely, if
${F}:\T\to\Hilb$ is an embedded tensor $C^*$-category with subobjects and direct sums and $r: \Rep(G)\to\T$ is a tensor $^*$-functor such that the following diagram
\[\xymatrix{
\Rep(G) \ar[rd]_{H} \ar[r]^r
&\T \ar[d]^{F} \\
 &\Hilb} \]
commutes and such that any irreducible of $\T$ is a subobject of some $r(u)$, then there is a quantum subgroup $K$ such that ($\T$, $F$) identifies with the pair corresponding to $K$ and $r$ with the restriction functor. The subgroup is unique up to the choice of the norm completion of the dense Hopf subalgebra.

Let $K=(Q_K,\Delta_K)$ be a quantum subgroup of $G=(Q_G,\Delta_G)$ defined by $\pi:Q_G\to Q_K$.
We may consider the
right translation of $G$ by $K$,
$$\rho:=1\otimes\pi\circ\Delta_G: Q_G\to Q_G\otimes Q_K,$$
which is an action of $K$ on $G$, in that it satisfies the relation
$$\rho\otimes 1\circ\rho= 1\otimes\Delta_K\circ\rho.$$
We may also consider the left translation of $G$ by $K$,
$$\lambda:=\pi\otimes1\circ\Delta_G: Q_G\to Q_K \otimes Q_G,$$
so that
$$1\otimes\lambda\circ\lambda=\Delta_K\otimes 1\circ\lambda.$$
This relation means that $\theta\circ\lambda$ is an action of $K$ on $G$ when endowed with the opposite comultiplication $\theta\circ\Delta_K$, where $\theta$ denotes the flip. We may thus consider the associated fixed point algebras
$$Q_{G/K}:=\{a\in Q_G:\rho(a)=a\otimes 1\},\qquad Q_{ K\backslash G}:=\{a\in Q_G:\lambda(a)=1\otimes a\},$$
which are analogues of the spaces of right and left $K$--invariant functions, respectively, and also
the analogue of the space of bi-$K$-invariant functions:
$$Q_{K\backslash G/K}:=Q_{K\backslash G}\cap Q_{ G/K}.$$

It is well known that $Q_{G/K}$ and $Q_{K\backslash G}$ are globally invariant under the translation action of $G$, in the sense that  if  $\Delta=\Delta_G$,
$$\Delta(Q_{K\backslash G})\subset Q_{K\backslash G}\otimes Q_G,\qquad \Delta( Q_{ G/K})\subset Q_G\otimes Q_{ G/K}.$$
For example, the first inclusion follows from
$$\lambda\otimes 1(\Delta(a))=\pi\otimes 1\otimes 1\circ\Delta\otimes 1\circ\Delta(a)=1\otimes\Delta \circ\lambda(a)=1\otimes\Delta(a)$$
for $a\in Q_{K\backslash G}$.

For the space of bi-$K$-invariant elements,
\begin{equation}\label{biK}
\Delta(Q_{K\backslash G/K})\subset Q_{ K\backslash G}\otimes Q_{G/K}.
\end{equation}

\medskip

\subsection{Normal quantum subgroups}\
\medskip

\noindent The notion of normal subgroup for compact quantum groups, as well as the following result, have been put forth by Wang (see, e.g., \cite{Wang} and references therein). We shall include a brief proof.
\begin{prop}\label{normal}
Let $K$ be a quantum subgroup of $G$. The following properties are equivalent,
\begin{itemize}
\item[{\rm a)}]
$Q_{K\backslash G}=Q_{ G/K},$
\item[{\rm b)}] $\Delta(Q_{K\backslash G})\subset Q_{K\backslash G}\otimes Q_{K\backslash G}$.
\item[{\rm c)}] If $v$ is an irreducible representation of $G$ such that the restricted representation $v\upharpoonright_K$ contains non-trivial invariant vectors, then $v\upharpoonright_K$ is a multiple of the trivial representation.
\end{itemize}
\end{prop}
\begin{proof}
a)$\Rightarrow$ b) follows from \eqref{biK}.\\
b)$\Rightarrow$ c) If $\psi$ is a non-trivial invariant vector for $v\upharpoonright_K$, and $(\phi_i)$ is an orthonormal basis, all coefficients
$$v_{\psi,\phi_i}:=\psi^*\otimes 1\circ v \circ \phi_i\otimes 1$$
lie in $Q_{K\backslash G}$ and
$$\Delta(v_{\psi,\phi_i})=\sum_jv_{\psi,\phi_j}\otimes v_{\phi_j,\phi_i}.$$
Then $v_{\phi_j,\phi_i}$ must be an element of $Q_{K\backslash G}$ for all $j$, whence all the $\phi_j$ are invariant vectors under $v\upharpoonright_K$.\\
c) $\Rightarrow$ a) follows from the fact that for any quantum subgroup $K$, $Q_{K\backslash G}$ is generated as a Banach space
by matrix coefficients $v_{\psi,\phi}$ of irreducible representations, where $\psi$ is invariant for $v\upharpoonright_K$ and $\phi$ is arbitrary. Similarly, $Q_{G/K}$ is generated by $v_{\phi',\psi'}$, where $\psi'$ is invariant and $\phi'$ is arbitrary.
\end{proof}
\begin{defn}
A quantum subgroup $K$ of $G$ is {\em normal} if it satisfies the equivalent conditions of Proposition \ref{normal}.
\end{defn}
Hence if $K$ is normal, $Q_{K\backslash G}$ becomes a Hopf $C^*$-subalgebra of $Q_G$ with respect to the restriction of the comultiplication of $G$. Moreover,
$K\backslash G=(Q_{K\backslash G}, \Delta)$ is a compact quantum group.
\begin{rem}
Wang showed that the above notion of normality can be equivalently stated in terms
of the   quantum adjoint action \cite{Wang_adjoint}. However, we shall not need it in this paper.
 \end{rem}
\begin{ex}
If $G=C^*(\Gamma)$ arises from a discrete group $\Gamma$, irreducible representations of $G$ are in one-to-one correspondence with elements from $\Gamma$, and are all one dimensional. This implies that the restriction of any irreducible to a quantum subgroup $K$ is still irreducible, so $K$ arises from a discrete group as well. In particular, Proposition \ref{normal}c holds, hence any quantum subgroup of $G$ is normal. The restriction functor gives rise to a group epimorphism from $\Gamma$ onto that group. If $\Lambda$ is the kernel, choosing the maximal norm, $K=C^*(\Lambda\backslash\Gamma)$ and  $K\backslash G=C^*(\Lambda)$.
\end{ex}

Let $K$ be a quantum subgroup of a compact quantum group $G$. If $K$ is cocommutative
then the restriction of  every irreducible representation $u$ of $G$ to $K$ is obviously direct sum
of $1$--dimensional representations of $K$, and this is in fact a characterization of cocommutativity of $K$. In later sections  a special class of cocommutative quantum subgroups will emerge,
those for which every $u\upharpoonright_K$ is a multiple of a single $1$-dimensional
representation. In the classical case, this property is a representation theoretic characterization of the property that
$K$ is a closed subgroup  of the {\em center} of $G$, by Clifford theorem
\cite{Weyl}.

On the other hand, in the noncommutative case, Wang has introduced the following notion
of central quantum subgroup.

\begin{defn} (\cite{Wang})
A quantum subgroup $K=(Q_K,\Delta_K)$ of a compact quantum group $G=(Q_G,\Delta_G)$ defined by
$\pi: Q_G\to Q_K$ is  called {\em central} if $\pi\otimes 1\Delta_G=\pi\otimes1\Delta'_G$
where $\Delta'_G$ is the coproduct of $Q_G$ opposite to $\Delta_G$.
\end{defn}

I. Patri has shown that Wang's  notion of central quantum subgroup  is in fact equivalent
to the previous  representation theoretic notion.    We are grateful to him for informing us of the following result.

\begin{prop} (\cite{patri}) $K$ is a central quantum subgroup of $G$ if and only if
 $u\upharpoonright_K$ is
a multiple of  a one-dimensional representation of $K$ for every irreducible representation $u\in\Rep(G)$.
\end{prop}

It follows that a central quantum subgroup  is not only cocommutative but also normal.
If $G$ is cocommutative
 every quantum subgroup is central.

\section{Quotient quantum groups}

The main topic of this section is a non-commutative analogue of the notion of quotient quantum group. In the classical theory, epimorphisms can be equivalently described by closed normal subgroups, the associated kernels. As in the non-commutative case not every embedded $G$-action is a quotient by a quantum subgroup,
 we introduce  quotient quantum groups without reference to subgroups. We thus start by recalling the relevant results. Later on, we characterize cases where quotients are induced by normal quantum subgroups and give sufficient conditions for their existence.

\subsection{A characterization of quotients by quantum subgroups}\
\medskip

\noindent Let $G=(Q,\Delta)$ be a compact quantum group. An action of $G$
on a unital $C^*$-algebra $A$ is a unital $^*$-homomorphism
$$\eta:  A\to  A\otimes Q$$ satisfying
$\eta\otimes 1\circ\eta=1\otimes\Delta\circ\eta$ and such that $\eta(A) \cdot (\Cset\otimes Q)$  is dense. This condition ensures that the linear subspace  $\A$ generated by the spectral subspaces (subspaces which transform like the irreducible representations of $G$ under the action) is dense \cite{Podles}.  Moreover, $\A$ is a $^*$-subalgebra invariant under the action of the dense Hopf algebra,
$$\eta(\A)\subset\A\odot\Q.$$
We shall say that the action $({A},\eta)$ is {\em embedded into the translation action}, or just {\em embedded}, if it is endowed with an injective $^*$-homomorphism
$$\alpha: \A\to {Q}$$ such that $\Delta\circ\alpha=\alpha\otimes 1\circ\eta$ on $\A$.
One  necessarily has $\alpha(\A)\subset\Q$. Hence, regarding $\A$ as a subalgebra of $\Q$, it becomes a translation invariant subalgebra,
$$\Delta(\A)\subset \A\odot\Q.$$
Note that this is an algebraic requirement, in that we are not requiring that $A$ can be embedded as a $C^*$-subalgebra of $Q$, although this will be automatically satisfied for example if the action $\eta$ of $G$ on $A$ is ergodic and coamenable. We shall refrain from giving details of this fact, as it will not be used in this paper; however, we shall later discuss the special case where $(A, \eta)$ is a Hopf $C^*$-algebra, see Propositions \ref{amenability1} and \ref{amenability2}.
Quotient spaces by quantum subgroups are clearly examples of embedded actions.

On the other hand, if $(A, \eta)$ is an embedded action, $\A$ must be generated by the coefficients  $u_{k,\psi_i}$ of unitary irreducible spectral representations of $G$, where $(\psi_i)$ is an orthonormal basis of $H_u$ and $k$ varies in a suitable subspace  $K_u\subset H_u$ whose dimension equals the multiplicity of $u$. In fact,
$$\Rep(G)\ni u \mapsto K_u\in\Hilb$$ extends additively to reducible representations, and one has $TK_u\subset K_v$ for $T\in(u,v)$, hence $u\to K_u$ becomes a $^*$-functor  \cite{PR}.
 For example, in the case of right quotients $(Q_{ K\backslash G}, \Delta)$, with $K$ a quantum subgroup,   $K_u$ is the space of invariant vectors for the restriction of $u$ to $K$.  However, not all embedded actions arise from quantum subgroups, examples
 can be found in \cite{Wang_ergodic} and \cite{Tomatsu}, see also \cite[p. 399]{PR}.
A characterization of quotient spaces by quantum subgroups among (not necessarily embedded) ergodic actions has been obtained in \cite{PR}. An analogous result in the embedded case has been idependently obtained in \cite{Tomatsu}.
We shall need here the special case of embedded actions. The following result has been  essentially proved in \cite[Sections 4, 5, 10]{PR}.
We sketch a proof as we need a slightly different formulation.
 \medskip

\begin{thm}\label{caratt}
Let $G$ be a compact quantum group and let $(A,\eta)$ be an embedded action of $G$. There exists a quantum subgroup $K$ of $G$ such that $(\A,\eta)$ is isomorphic to the dense algebraic action of $G$ on $K\backslash G$ if and only if for each pair of irreducible representations $u$, $v$ of $G$, the spectral spaces satisfy
\begin{equation}\label{3.1}
(1_{\overline{u}}\otimes K_v\otimes 1_u) R \subset K_{\overline{u}vu},\quad R\in(\iota, \overline{u}u).
\end{equation}
The subgroup is unique up to the choice of the norm completion on the dense Hopf subalgebra and its representation category is determined by
\begin{equation}\label{3.2}
(u\upharpoonright_K, v\upharpoonright_K)=\overline{R}^*\otimes 1_v\circ 1_u\otimes K_{\overline{u}v},
\end{equation}
where $u,v\in\Rep(G)$ are irreducible and $\overline{R}\in(\iota, u\overline{u})$ is non-zero.
\end{thm}
\begin{proof}
The necessity of the condition is a consequence of the fact that
$\Rep(K)$ is a tensor category and restricting a representation to $K$ defines a tensor functor. Indeed, $K_u:=(\iota, u\upharpoonright_K)$,
hence
$$(\iota, \overline{u}u)\subset (\iota, (\overline{u}u)\upharpoonright_K)= (\iota, \overline{u}\upharpoonright_Ku\upharpoonright_K),$$
so for $R\in (\iota, \overline{u}u)$,
$$(1_{\overline{u}}\otimes K_v\otimes 1_u) R\subset
(\iota, \overline{u}\upharpoonright_K v\upharpoonright_Ku\upharpoonright_K)=
(\iota, (\overline{u} vu)\upharpoonright_K)=K_{\overline{u} vu}.$$
Conversely, if the $K_u$ are the spectral spaces of an embedded action of $G$
then one can use Frobenius reciprocity to construct the representation category
of a quantum subgroup of $K$ starting with its invariant vectors for the restricted representations. Explicitly, one can show that for (possibly reducible) $u$, $v\in\Rep(G)$, the subspaces of $\B(H_u, H_v)$ given by formula \eqref{3.2}, where now   $\overline{R}$ defines a conjugate for $u$ in $\Rep(G)$,
form an embedded tensor $C^*$-category containing $\Rep(G)$ as a subcategory.  Specifically, condition \eqref{3.1}, together with the fact that
$u\in\Rep(G)\to K_u\in\Hilb$ is a functor, play a role in the proof of tensoriality of this category.
Hence this category, after completion with subobjects and direct sums, is the representation category of
a quantum subgroup $K$ of $G$ having $K_u$ as fixed vectors for $u\upharpoonright_K$.
\end{proof}

\subsection{Quotient quantum groups}\
\medskip

\noindent Let $G$ and $L$ be compact quantum groups with associated Hopf $C^*$-algebras $ {Q}_G$ and ${Q}_L$ respectively. An injective homomorphism ${Q}_L\to{Q}_G$ of unital Hopf $C^*$-algebras restricts to the dense subalgebras $ \Q_L\to\Q_G$ since $\varphi$ takes representations of $L$ to representations of $G$.
On the other hand, an injective homomorphism of Hopf $^*$-subalgebras $\varphi: \Q_L\to\Q_G$ may not extend in an injective way to the completions.
For example, we may choose for $G$ the maximal completion of a given non-coamenable compact quantum group, and for $L$ the reduced completion.
However,  lack of  coamenability is the only obstruction.
\begin{prop}\label{amenability1}
Let $G$ and $L$ be compact quantum groups and
let $\varphi:\Q_L\to\Q_G$ be an injective $^*$-homomorphism of the associated
dense Hopf $^*$-subalgebras. If $L$ is coamenable, $\varphi$ extends to an isometric homomorphism
between the completed Hopf $C^*$-algebras.
\end{prop}
\begin{proof}
Let us regard $\Q_L$ as a Hopf $^*$-subalgebra of $\Q_G$.
The restriction of the Haar measure $h_G$ of $G$ to $\Q_L$ is the Haar measure
$h_L$ of $L$, hence $L^2(G)$ contains a copy of $L^2(L)$.  Moreover, the GNS representation $\pi_{h_G}$    restricts to the GNS representation $\pi_{h_L}$ of $\Q_L$ on that subspace. Therefore for $x\in\Q_L$,
$\|\pi_{h_L}(x)\|  \leq\|\pi_{h_G}(x)\|$. On the other hand, with respect to the maximal
norms of $\Q_G$ and $\Q_L$ we obviously have
$\|\pi_{h_G}(x)\|\leq \|x\|\leq\|x\|_{\text{max}}^G\leq  \|x\|_{\text{max}}^L$, where
$\|\cdot{}\|$ is the original norm of $Q_G$. If $L$ is coamenable, the reduced and maximal norms of $\Q_L$ coincide, so the original
norm of $Q_G$ restricts to the unique norm  of $\Q_L$.
\end{proof}

\begin{prop}\label{amenability2}
Let $L$ and $G$ be compact quantum groups such that the associated Hopf $C^*$-algebras are related by an injective inclusion,
$Q_L\to Q_G$. If $G$ is coamenable then $L$ is coamenable as well.
\end{prop}
\begin{proof}
By \cite[Theorem 2.2]{BMT}, a compact quantum group is coamenable if and only if the Haar measure is faithful and the counit is norm bounded.
On the other hand, the Haar state of $Q_G$ restricts to the Haar state of $Q_L$, and the counit of $\Q_G$ restricts to the counit of $\Q_L$.
\end{proof}

We shall mostly be interested in the case where $G$ is a given compact quantum group, and $L$ is the compact quantum group associated to a Hopf $C^*$-subalgebra of $Q_G$ obtained by selecting a family of representations. In this case, too, there is obviously no problem in extending uniquely the inclusion map to the completion. We shall adopt the following algebraic notion of epimorphism.

\begin{defn}
If $L$ and $G$ are compact quantum groups, a non-commutative
epimorphism $G\to L$ is an injective $^*$-homomorphism $\varphi: \Q_L\to\Q_G$ between the associated dense Hopf $^*$-algebras satisfying
$$\Delta_G\circ \varphi=\varphi\otimes\varphi \circ \Delta_L.$$
We shall refer to $L$ as a {\em quotient quantum group} of $G$.
\end{defn}
In other words,  a quotient quantum group is an embedded action which is also
a Hopf $C^*$-algebra. An epimorphism $G\to L$
gives rise to a commutative diagram,
\begin{equation}\label{epi}
\xymatrix{
\Rep(L) \ar[rd]_{H^L} \ar[r]
&\Rep(G) \ar[d]^{H^G} \\
&\Hilb}
\end{equation}
where the top arrow takes the representation $(u: H_u\to H_u\otimes { Q}_L)\in\Rep(L)$ to the representation $\hat{\varphi}(u):=1_{H_u}\otimes\varphi\circ u\in\Rep(G)$. Note that the range of $u$ is actually contained in $H_u\otimes \Q_L$. An arrow $T\in(u,v)$ of $\Rep(L)$ is a linear map  between the associated Hilbert spaces such that $T\otimes 1_{{Q}_L}\circ u=v\circ T$. The functor $\hat{\varphi}$ acts trivially on arrows.

\begin{prop}\label{categoriesquotients}
Let $G$ be a compact quantum group. The assignment
$$L \ni \varphi\mapsto\hat{\varphi} \in \Rep(L)$$
establishes a bijective correspondence between epimorphisms $G\to L$  of compact quantum groups and full tensor $^*$-functors $\S\to \Rep(G)$ of tensor $C^*$-categories with conjugates, subobjects and direct sums.
\end{prop}
\begin{proof}
The algebraic structure   of a compact quantum group is explicitly related to the algebraic structure of its representation category, and this relation makes the associated functor $\hat{\varphi}$ into a tensor $^*$-functor. If $T\in(\hat{\varphi}(u), \hat{\varphi}(v))$
then
$$1_{H_v}\otimes\varphi\circ T\otimes 1_{\Q_L}\circ u=
T\otimes 1_{\Q_G}\circ1_{H_u}\otimes\varphi\circ u=$$
$$T\otimes 1_{\Q_G}\hat{\varphi}(u)=\hat{\varphi}(v)T=1_{H_v}\otimes\varphi\circ v\circ T,$$
hence $T\in(u,v)$ if $\varphi$ is injective, showing that $\hat{\varphi}$ is full. The converse statement
is a consequence of the explicit reconstruction of the Hopf
algebra from an embedded category. A full tensor $^*$-functor   $F:
\Rep(L)\to\Rep(G)$ such that $H^G\circ F=H^L$
takes irreducible representations of $L$ into irreducible representations of $G$.
We thus have a map $\varphi_F$ taking elements of the subspace $\overline{H}_{u}\otimes H_u$
of $\Q_L$, with $u\in\Rep(L)$ irreducible, to itself, regarded as an element
of $\Q_G$ via the commutative diagram \eqref{epi}.   This map must preserve the Hopf $^*$-algebra operations and is injective since  the subspaces $\overline{H}_{u}\otimes H_u$ are in direct sum. One has $\widehat{\varphi_F}=F$ and $\varphi_{\hat{\varphi}}=\varphi$.

On the other hand, any full tensor $^*$-subcategory of $\Rep(G)$ with conjugates, is embeddable into the Hilbert spaces, hence, by duality, it corresponds to a compact quantum group, which is a quotient of $G$.
\end{proof}
\begin{ex}
If $K$ is a normal quantum subgroup of $G$ then $K\backslash G$ is a quotient quantum group of $G$ by Proposition \ref{normal}b.
\end{ex}
\subsection{Normal tensor subcategories}\label{normalsubsect}\
\medskip

\noindent In this section we give a characterization,  in terms of the associated inclusion $\S\subset \Rep(G)$, of quotient quantum groups which can be written as quotients by a normal quantum subgroup. To this aim, we establish a connection with Theorem \ref{caratt}.

We start with a full inclusion $\S\subset\T$ of abstract tensor $C^*$-categories with conjugates, subobjects and direct sums. An irreducible object of $\S$ stays irreducible in $\T$, hence a complete set of irreducible objects of $\T$ contains a complete set of irreducible objects of $\S$ as a subhypergroup. Let ${\S^\perp}$ be the full subcategory of $\T$ whose objects are those objects of $\T$ that are disjoint from all objects of $\S$. Note that $\S^\perp$ has conjugates, subobjects and direct sums, but generally fails to be tensorial. We may decompose every object $u$ of $\T$ as
$$u=u_\S\oplus u_{\S^\perp},$$
where $u_\S\in\S$ and $u_{\S^\perp}\in\S^\perp$, where $u_\S$ is the maximal subobject of $u$ lying in $\S$.
 Note that if $u=u_\S$ and $v=v_{\S^\perp}$ then $(u, v)=0$. Therefore any arrow $T\in(u_\S\oplus u_{\S^\perp}, v_\S\oplus v_{\S^\perp})$ takes a diagonal form,
$$T=T_\S\oplus T_{\S^\perp},$$
with $T_\S\in (u_\S, v_\S)$, $T_{\S^\perp}\in(u_{\S^\perp}, v_{\S^\perp})$. We may thus consider the functor $$S:\T\to\S,$$ defined by $u\mapsto u_\S$ on objects and $T\mapsto T_\S$ on arrows. This is obviously a $^*$-functor between tensor $C^*$-categories.
 \begin{lemma}\label{splitinj}
If $u \in \S, v \in \S^\perp$, then $uv, vu \in \S^\perp$.
\end{lemma}
\begin{proof}
Let $w \in \S$. By Frobenius reciprocity \eqref{frobenius}, we have
$$(uv, w) \simeq (v, \overline u w), \qquad (vu, w) \simeq (v, w \overline u).$$
However, $\overline u w, w \overline u$ both belong to $\S$ as $\S$ is a tensor category with conjugates. Hence $(uv, w), (vu, w)$ are both trivial for every choice of $w$, and we conclude that $uv, vu$ both lie in $\S^\perp$.
\end{proof}
By Lemma \ref{splitinj}, it follows that
$$(uv)_\S= u_\S v_\S\oplus(u_{\S^\perp} v_{\S^\perp})_\S,$$
$$(uv)_{\S^\perp}=u_\S v_{\S^\perp}\oplus u_{\S^\perp} v_\S\oplus
(u_{\S^\perp} v_{\S^\perp})_{\S^\perp},$$
for every $u, v \in \T$. Hence $S$ is not a tensor functor, as $(uv)_\S$ only contains $u_\S v_\S$ as a subobject. For example, if $u\in\S^\perp$, $u_\S=\overline{u}_\S=0$ while $(\overline{u} u)_\S$ contains the trivial object of $\S$.
\medskip

\begin{rem}
It is not difficult to show that the functor $S: u\to u_\S$ is a quasitensor functor in the sense of \cite{PR}.
\end{rem}
\begin{prop}\label{normalcat}
Let $\S\subset\Rep(G)$
be a full tensor $C^*$-category
with conjugates and  subobjects and $\Q_L$ the associated Hopf $^*$-algebra. For an irreducible $u=(u_{js})\in\Rep(G)$, the following conditions are equivalent,
\begin{itemize}
\item[{\rm a)}]
    $1_{\overline{u}}\otimes H_v\otimes 1_u\circ
    R\subset H_{(\overline{u}vu)_\S}, \quad R\in(\iota, \overline{u}u), v\in \S$ irreducible,
\item[{\rm b)}]
    $\sum_i u_{ij}^*xu_{i,s}\in\Q_L, \quad x\in\Q_L$.
\end{itemize}
\end{prop}
 We omit a detailed proof. We just note that this is a consequence of Tannakian reconstruction of the involution and product formula of the dense Hopf algebra  in terms of the Hilbert spaces of the representations, see Section \ref{tannaka}.

\begin{defn}
A full tensor $C^*$-category $\S\subset\Rep(G)$ with conjugates and subobjects will be called {\em normal} if the above equivalent conditions hold for any irreducible $u\in\S^\perp$.
\end{defn}

Note that the conditions of Proposition \ref{normalcat} are always satisfied by the objects $u\in{\mathcal S}$.
Hence, normality amounts to requiring that it is satisfied by all objects of $\Rep(G)$.

\begin{ex}
If $\Lambda\subset\Gamma$ is an inclusion of discrete groups, $C^*(\Lambda)$, with its natural comultiplication, is a quantum quotient of  $G=C^*(\Gamma)$. Here $\Lambda$ and $\Gamma-\Lambda$ identify, respectively, to the sets of irreducible objects of $\S$ and $\S^\perp$, so $u_\S=u$ if $u\in\Lambda$ and $u_\S=0$ otherwise.
 Since the product of irreducible objects is irreducible, the normality condition reduces to the requirement that $\Lambda$ is a normal subgroup of $\Gamma$.
\end{ex}
If the subgroup $\Lambda$ is central in $\Gamma$  then $\Lambda$ is normal.
We next discuss  sufficient conditions for normality of a tensor subcategory $\S$ of $\Rep(G)$.

\begin{prop}\label{sufficientfornormality}
Let $\S\subset\Rep(G)$ be a full tensor $C^*$-subcategory with conjugates, subobjects and direct sums, $\Q_L$ the associated quotient quantum subgroup and  $\Q_L^\perp$ the linear subspace of $\Q_G$ generated by the coefficients of the representations of $\S^\perp$. Consider the following properties,
\begin{itemize}
\item[{\rm a)}] $\Q_L$ and $\Q_L^\perp$ are in the commutant of each other,
\item[{\rm b)}] for $u\in\S^\perp$, $v\in\S$ the permutation operator   $\vartheta_{v,u}: H_v\otimes H_u\to H_u\otimes H_v$ is an arrow in $(vu, uv)$.
\item[{\rm c)}]  for $u\in\S^\perp$, $v\in\S$, there is an arrow $\varepsilon_{v,u}\in(vu, uv)$ such that
$$(\varepsilon_{v,\overline{u}}
\otimes 1_u)\phi\otimes R= 1_{\overline{u}}\otimes\phi\otimes 1_u\circ R,$$
for $\phi\in H_v$, $R\in(\iota, \overline{u}u)$,
\item[{\rm d)}] $\overline{u}vu\in\S$,  $\quad u\in\S^\perp,  v\in \S$ irreducible.
\end{itemize}
Then
 a)$\Leftrightarrow$ b)$\Rightarrow$ c) and
  any of a), b), c), d) implies that $\S$ is normal.
  \end{prop}
\begin{proof}
The equivalence of a) and b) follows again from  Tannaka duality, and obviously they imply c). We check normality if c) holds. If $\phi\in H_v$, and $u$, $v$ are as required,
$$1_{\overline{u}}\otimes\phi\otimes 1_u\circ R=(\varepsilon_{v,\overline{u}}\otimes 1_u)\phi\otimes R\subset (\varepsilon_{v,\overline{u}}\otimes 1_u)H_v\otimes H_{(\overline{u}u)_\S}$$
$$\subset (\varepsilon_{v,\overline{u}}\otimes 1_u) H_{(v\overline{u}u)_\S}\subset H_{{{(\overline{u}vu)}}_\S},$$
where we have used the fact that $\overline{u}\in\S^\perp$, $\iota\in\S$, that $u\mapsto u_\S$ is a functor and that $u_\S v_\S$ is contained in $(uv)_\S$. The fact that d) implies normality follows from Proposition \ref{normalcat}.
 \end{proof}

\begin{rem}
In terms of matrix coefficients, the  above condition   $\overline{u}vu\in\S$
is equivalent to  $\overline{u}_{\psi,\varphi}v_{\xi,\eta }u_{\psi',\varphi'}\in{\mathcal Q}_L$.
\end{rem}

 We are now ready to prove the following application of Theorem \ref{caratt} to quotient quantum groups.
\begin{thm}\label{carattnormal}
Let $L$ be a quotient quantum group of $G$ and  $\S=\Rep(L)$ the corresponding subcategory of $ \Rep(G)$. There is a normal quantum subgroup $K$ of $G$ such that $(\Q_L,\Delta_L)$ is isomorphic to the dense Hopf $^*$-subalgebra of $K\backslash G$ if and only if $\S$ is normal. It is unique   up to the choice of the norm completion on the dense Hopf subalgebra, and its representation category is determined by
$$(u\upharpoonright_{K}, v\upharpoonright_{K})=\{\overline{R}^*\otimes 1_v\circ 1_u\otimes\phi, \quad \phi\in H_{({\overline{u} v})_\S}\}$$
where $u,v\in\Rep(G)$ are irreducible and $\overline{R}\in(\iota, u\overline{u})$ is non-zero. In particular,
$$\dim(u\upharpoonright_{K}, v\upharpoonright_{K})=\dim H_{({\overline{u} v})_\S}.$$
\end{thm}
\begin{proof}
If we regard the comultiplication of $L$ as an action of $G$ of ${Q_L}$, $L$ becomes an embedded action of $G$. Moreover, if a quantum subgroup $K$ realizes $L$ as a quotient $G$-action, then Proposition \ref{normal}b shows that $K$ is automatically normal since $L$ is a quantum group. We are thus reduced to apply Theorem \ref{caratt}. The spectral functor of this action is the functor
$$K: \Rep(G)\stackrel S\rightarrow\S\stackrel H\rightarrow\Hilb$$
obtained by composing $S$
 with the embedding of $\S$ in the Hilbert spaces, hence in particular $K_u=H_{u_\S}$. We claim that it suffices to verify the required property for irreducible representations
$v\in\S$, $u\in\S^\perp$. Indeed, for $v\in\S^\perp$, $K_v=0$. Moreover, for $u, v\in \S$, $\overline{u}vu\in\S$, so $K_{\overline{u}vu}$ is the whole Hilbert space and the required property is trivially satisfied.
\end{proof}

\begin{defn}
A tensor subcategory ${\mathcal S}\subset\Rep(G)$ satisfying condition c) of Proposition  \ref{sufficientfornormality}, or the associated quantum subgroup of $G$,  will be referred to as being {\em strongly normal}.
\end{defn}

For example ${\mathcal S}=\langle \iota\rangle$ --- the subcategory of $\Rep(G)$ whose only objects are multiples of the trivial representation --- and  ${\mathcal S}=\Rep(G)$ are normal and correspond to $K=G$ and the trivial subgroup, respectively.

Note that any object $v\in\S$ restricts to a multiple of the  trivial representation since $(\iota, v\upharpoonright_K)$ has full dimension, whereas $(\iota, v\upharpoonright_K)=0$ for $v\in\S^\perp$.

\begin{prop}\label{centrality}
Condition d) of Proposition  \ref{sufficientfornormality} is equivalent to centrality of the  quantum subgroup associated to ${\mathcal S}$.
\end{prop}
\begin{proof}
If $K$ is the quantum subgroup defined by a normal subcategory ${\mathcal S}$
then d) means that for $v$, $u$ are required,   $\overline{u}vu$ restricts to a multiple of the trivial representation of
$K$. Since the same holds for $v$, this is equivalent to requiring $\overline{u}u$ the same property. Frobenius reciprocity and Peter-Weyl theory show that this is equivalent to $K$ being central.
\end{proof}

 \begin{ex}\label{su2}
Let $G={\rm SU}_q(2)$, and denote by  $u_n$ the (self-conjugate) irreducible representation of dimension $n+1$. Consider the full subcategories $\S$ and $\S^\perp$ of $\Rep(G)$ with subobjects and direct sums generated by the irreducible representations with even and odd indices respectively.  The Clebsch-Gordan fusion rules show that $\S$ is a tensor $C^*$-subcategory with conjugates.
It is indeed the category of representations of a quantum ${\rm SO}(3)$.

Proposition \ref{sufficientfornormality}d holds, hence $\S$ is a normal subcategory, and, by Theorem \ref{carattnormal}, there must exist a normal quantum subgroup $K$ inducing the quotient, which is central by Proposition \ref{centrality}. Since ${u}_1^2\in\S$, by Frobenius reciprocity $(u_1\upharpoonright_K, u_1\upharpoonright_K)$ has full dimension, hence $u_1\upharpoonright_K$ is direct sum of two  one-dimensional representations, $g$ and $g'$, which are non-trivial since $(\iota, u_1\upharpoonright_K)=0$. Since $u_1^2$ restricts to the trivial representation, $g'=g^{-1}$ and $g^2=1$. Therefore $K\simeq \Cyc_2$ is the  cyclic group\footnote{Henceforth, we will denote by $\Cyc_n, 1 < n \leq \infty$ the cyclic group of order $n$.} of order $2$.
\end{ex}

 \section{The identity component of a compact quantum group}\label{idcompsect}

In this section we introduce the identity component $\go$ of a compact quantum group $G$ starting from the notion of connectedness introduced by Wang in \cite{Wang}. We next introduce totally disconnected compact quantum groups as those for which $\go$ is trivial, and, looking at examples arising from discrete groups, we discuss the main novelties with respect to the classical case.
\begin{defn}[\cite{Wang}] A compact quantum group is {\em connected} if the associated Hopf $C^*$-algebra admits no finite-dimensional unital Hopf $^*$-subalgebra other than the trivial one.
\end{defn}
In the classical case this definition says that the only finite group $\Gamma$ for which there is a continuous epimorphism $G\to\Gamma$ is the trivial group. This is obviously weaker than connectedness, but it is in fact equivalent since if $G$ is disconnected, we have a non-trivial compact component group $\go\backslash G$, which is profinite. Hence it has non-trivial finite quotients. We next consider the categorical counterpart of connectedness.

\subsection{Torsion in tensor $C^*$-categories}\
\medskip

\noindent
 \begin{defn}
An object $u$ of a tensor $C^*$-category with conjugates $\T$ will be called a {\it torsion object} if the smallest full tensor $C^*$-subcategory  $\T_u$ of $\T$ with conjugates and subobjects containing $u$ has finitely many inequivalent irreducible objects.
\end{defn}

\begin{prop}\label{subconj}
If $u$ is a torsion object, so is
every subobject of $u$,  the  conjugate of $u$ or
any finite direct sum of  objects of  $\T_u$.
\end{prop}

Tensor products of torsion objects may fail to be torsion.
For example,  consider the representation category of the compact quantum group arising from a discrete group $\Gamma$. The  set of irreducible torsion representations corresponds to the set $\Gamma^t$ of torsion elements of $\Gamma$ and this is not a subgroup, in general. An example is provided by the infinite dihedral group $\Dih_\infty = \Cyc_2*\Cyc_2$.

  \begin{defn}
  An abstract  tensor $C^*$-category $\T$ with conjugates and subobjects admitting  no non-trivial irreducible torsion object, will be called {\it torsion-free}.
\end{defn}

\begin{prop}
A compact quantum group $G$ is connected if and only if
$\Rep(G)$ admits no non-trivial  full tensor $C^*$-subcategory with conjugates and finitely many irreducible representations.
Equivalently, $\Rep(G)$ is torsion-free.
 \end{prop}
\begin{proof}
Since the irreducible components
of a torsion object are torsion objects,
if a category admits a non-trivial torsion object, then it also admits a non-trivial irreducible torsion object.
\end{proof}

In particular, quantum groups with fusion rules identical (or quasi-equivalent) to those of connected compact groups are connected.

\begin{exs}\quad
\begin{itemize}
\item[{\rm a)}] Finite non-trivial quantum groups are clearly disconnected.
\item[{\rm b)}] If $G$ arises from a discrete group $\Gamma$, the irreducible torsion objects of $\Rep(G)$ correspond to the elements of the torsion subset $\Gamma^t$ of $\Gamma$, hence $G$ is connected if and only if $\Gamma$ is torsion--free.
\item[{\rm c)}] The deformation quantum groups $G_q$ obtained from classical compact Lie group, as well as $A_o(F)$, are connected, as the fusion rules are the same as those of the classical groups.
\item[{\rm d)}] Inspection of the fusion rules \cite{Banica} of $A_u(F)$ shows that these quantum groups are connected as well.
\end{itemize}
\end{exs}
In the following proposition, we use the notion of image of a quantum subgroup $K$ of $G$ in a quotient $L$ of $G$, as introduced in the Appendix.
\begin{prop}
Let $G$ be a compact quantum group.
\begin{itemize}
\item[{\rm a)}] If $G$ is connected, any quotient quantum group $L$ of $G$ is connected.
\item[{\rm b)}] Let  $K$ and $L$ be a quantum subgroup and quotient of $G$ respectively.
If $K$ is connected, the image of $K$ in $L$ is connected.
\end{itemize}
\end{prop}
\begin{proof}
a) follows from the fact that the representation category of $L$ is just a full subcategory of the representation category of $G$.

b) The image of $K$ in $L$ is a quotient quantum group of $K$, hence b) follows from a).
\end{proof}

 \begin{prop}
If $\T\subset \U$ is an inclusion of tensor $C^*$-categories with conjugates and subobjects, then every torsion object of $\T$ is torsion in $\U$.
\end{prop}
\begin{proof}
If $u\in\T$ is a torsion object, it generates a tensor $C^*$-subcategory of $\T$ (full, with conjugates, subobjects and direct sums) with a finite set, say $F$, of irreducible objects. As an element of $\U$, every object of $F$ decomposes into a finite direct sum of inequivalent irreducible representations of $\U$ with suitable multiplicities. Hence, as an object of $\U$, $u$ generates a tensor $C^*$-subcategory of $\U$ with finitely many irreducible representations.
\end{proof}
In particular, choosing for $\Rep(G)\subset \Rep(K)$ the inclusion given by restricting a representation of $G$ to a quantum subgroup $K$, gives the following useful result.
\begin{cor}\label{restrictingtorsiontosubgroup}
If $K$ is a  quantum subgroup of a compact quantum group $G$, every torsion representation $u$ of $G$ restricts to a torsion representation of $K$. In particular, if $u\in \Rep(G)$ is torsion and $K$ is connected, then $u\upharpoonright_K$ is a multiple of the trivial representation.
\end{cor}

\subsection{The identity component $G^\circ$   and the normal counterpart $G^{\rm n}$}\
\medskip

\noindent
 Let us identify $\Rep(G)$ with a tensor $C^*$-subcategory of $\Hilb$ with subobjects and direct sums, via   the  embedding functor $H:\Rep(G)\to\Hilb$.
 Consider the subcategory  $\T^\circ\subset \Hilb$
with arrows between the objects $u$, $v\in\Rep(G)$  given by
$$(u,v)_{\T^\circ}=\cap_K(u\upharpoonright_K, v\upharpoonright_K),$$
where the intersection is taken over all the connected quantum subgroups
$K$  of
$G$.  $\T^\circ$ is clearly a tensor $^*$-subcategory of $\Hilb$ containing in turn
$\Rep(G)$ as a tensor $^*$-subcategory and with the same objects. Completing
$\T^\circ$ under subobjects and direct sums gives the representation category
of a quantum subgroup $\go$ of $G$. Note that in the case where $G$ is a compact group, this construction yields the closed subgroup generated by the connected
closed subgroups of $G$, i.e., the connected component of the identity of $G$.

\begin{prop}\label{conncomp}
$\go$ is the largest connected quantum subgroup of $G$.
\end{prop}
\begin{proof}
Note that $\go$ contains every connected quantum subgroup $K$ as a quantum subgroup by construction. We are left to show that $\go$ is connected. Let $v$ be an irreducible torsion object of $\Rep(\go)$ and let $u$ be an irreducible object of $\Rep(G)$ such that $v<u\upharpoonright_{\go}$. The orthogonal projection $E_v\in(u\upharpoonright_{\go}, u\upharpoonright_{\go})$ corresponding to $v$ is an arrow in every $(u\upharpoonright_{K}, u\upharpoonright_{K})$ and it corresponds to  $v\upharpoonright_{K}$. Restriction of a torsion object to a quantum subgroup is still torsion, so $v\upharpoonright_K$ is a multiple of the trivial representation of $K$ since $K$ is connected. Hence elements of an orthonormal basis of the range of $E_v$ lie in every arrow space $(\iota, v\upharpoonright_K)\subset (\iota, u\upharpoonright_K)$, hence they lie in $(\iota, u\upharpoonright_{\go})$. This shows that $v$ is a multiple of the trivial representation of $\go$.
\end{proof}
\begin{defn}
We shall refer to $\go$ as the {\it  identity component} of $G$. If  $\go$ is the trivial group, $G$ will be called {\em totally disconnected}.
\end{defn}
\begin{rem}
Clearly, $G=\go$ if and only if $G$ is connected.
\end{rem}
A connected quantum subgroup $K$ of $G$ is a subgroup of $\go$ by construction, hence there is  a commutative diagram
\[\xymatrix{
\Rep(\go) \ar[rd]_{H} \ar[r]
&\Rep(K) \ar[d]^{H} \\
&\Hilb}\]
where the top arrow is the restriction functor. Conversely, if a connected quantum subgroup $G'$ of $G$ has associated commutative  diagrams for each connected quantum subgroup  $K$ of $G$ then $G'=\go$. Summarizing, $\go$ is the connected quantum subgroup of $G$ defined by the following universal property for connected quantum subgroups $K$ of $G$,
\[\xymatrix{
\Rep(G) \ar[d] \ar[r] &\Rep(\go)\ar[ld] \ar[d]\\
\Hilb  &\ar[l] \Rep(K)}\]

\begin{rem} The notion of identity component of a quantum group is often implicitly used in representation theory   to rule out certain finite-dimensional representations. The simplest instance is that of $U_q(\mathfrak{su}_{2})$ for $0<q<1$, where restricting to the identity component amounts to focusing on the so-called ``type $I$ representations'' --- those representations with positive weights (see, e.g., \cite{CP}).
We shall discuss this in more detail in the last section, where we shall also consider the case
of negative parameters.
\end{rem}

We shall  often need the following fact,  an easy consequence of  Corollary \ref{restrictingtorsiontosubgroup}.

\begin{prop}\label{restrictingtorsiontoconnectedcomponent}
Every torsion object of $\Rep(G)$ restricts to a multiple of the trivial representation of $\go$.
\end{prop}

In the classical context of compact  groups, the converse of Proposition \ref{restrictingtorsiontoconnectedcomponent} holds by profiniteness of the component group. This property fails for compact quantum groups. Indeed,
a tensor product of (even irreducible) torsion representations may fail to be torsion, however it still  restricts to a multiple of the trivial representation of $\go$.

\begin{cor}\label{subgroupquotient}
If $N$ is a normal quantum subgroup of $G$ such that $N$ and $N\backslash G$ are connected then $G$ is connected.
\end{cor}

\begin{proof}
If $u$ is a torsion representation of $G$ then it restricts to a multiple of the trivial representation on $\go$, as well as on every connected quantum subgroup, hence in particular on $N$. Therefore $u$ is actually a representation of the quotient quantum group $N\backslash G$ which is connected by assumption, hence $u$ must be a multiple of the trivial representation.
\end{proof}

\begin{rem}\label{totdisc3}
The problem of total disconnectedness of the quantum component group will be treated in the next section.
One would be tempted  to use Corollary \ref{subgroupquotient} in order to conclude that if $\go$ is normal, then $\go\backslash G$ is totally disconnected. However, the classical argument requires that subgroups $F$ of a quotient $N\backslash G$ arise as images of subgroups of $G$.
This fact does not hold in general. We will discuss this in more detail later in the Appendix, where we shall also address some special cases where the two notions of subquotients coincide.
\end{rem}

We next discuss the special case of cocommutative quantum groups.
\begin{ex}\label{exconnectedcomponent}
If $G=C^*(\Gamma)$ is the quantum group associated to the discrete group $\Gamma$, a connected quantum subgroup $K$ is associated to a torsion-free quotient $\Lambda\backslash\Gamma$ by a normal subgroup $\Lambda$. The identity component $\go$ corresponds to the universal torsion-free quotient $\Lambda^\circ\backslash\Gamma$, where $\Lambda^\circ$ is the torsion-free radical of $\Gamma$ in the sense of \cite{BH}, i.e., the intersection of all normal subgroups with torsion-free quotient.

Note that $\Lambda^\circ$ contains the torsion subset $\Gamma^t$. If $\Gamma^t$ is a subgroup of $\Lambda$, then it is normal and $\Gamma^t\backslash\Gamma$ is torsion free, hence $\Lambda^\circ=\Gamma^t$. In this case, $\go$ corresponds to $\Gamma^t\backslash\Gamma$ and $\go\backslash G$ to $\Gamma^t$. In particular,  $\go\backslash G$ is totally disconnected since $(\Gamma^t)^t=\Gamma^t$.

In the general case, $\Lambda^\circ$ contains the subgroup $N_1$ generated by $\Gamma^t$, which is normal.  In \cite{Chiodo}, Chiodo and Vyas give an example of a finitely presented group for which $N_1\backslash\Gamma$ is isomorphic to a non-trivial torsion (cyclic) group. This shows that, in general, $\Lambda^\circ$ strictly contains $N_1$; we will generalize this example in Section \ref{normsect}.
\end{ex}

Consider now the quantum subgroup $G^{\rm n}$ of $G$ whose representation category is determined by
$$(u\upharpoonright_{G^{\rm n}}, v\upharpoonright_{G^{\rm n}}):=\cap_{N}(u\upharpoonright_{N}, v\upharpoonright_{N}),$$
for $u$, $v\in\Rep(G)$, where $N$ ranges over all normal connected quantum subgroups of $G$.

\begin{prop} $G^{\rm n}$ is  the largest connected,  normal quantum subgroup of $G$, and $G^{\rm n}\subset \go$.
\end{prop}
\begin{proof} The same arguments as in Proposition \ref{conncomp} show that $G^{\rm n}$ is connected. In particular, $G^{\rm n}\subset \go$. If $v$ is irreducible and the arrow space $(\iota, v\upharpoonright_{G^{\rm n}})$ is not zero then for every $N$, $(\iota, v\upharpoonright_{N})$ is non-zero, hence it is full since $N$ is normal. Therefore $(\iota, v\upharpoonright_{G^{\rm n}})$ is full as well, hence $G^{\rm n}$ is normal.
\end{proof}

Notice that $\go=G^{\rm n}$ if and only if $\go$ is normal. In this case, the quotient  $\go\backslash G$ will be called the {\it quantum  component group}. We shall give a description of $\go$ and $G^{\text{n}}$ as limits of certain transfinite sequences defined by torsion in Section \ref{normsect}.

\subsection{Totally disconnected quantum groups}\label{totdiscqg}\

\begin{prop}\label{sufficientfortotdisc}
If every irreducible representation of $\Rep(G)$ is a torsion object,
then $G$ is totally disconnected.
\end{prop}
\begin{proof}
Every irreducible representation $v$ of $\go$ is a subrepresentation
of the restriction of an irreducible representation of $G$, which is assumed to be torsion, hence $v$ is trivial by Proposition \ref{restrictingtorsiontoconnectedcomponent}. This shows that $\go$ is trivial.
\end{proof}
\begin{exs}\label{examples}\quad
\begin{itemize}
\item[{\rm a)}] Finite quantum groups are  clearly totally disconnected.

\item[{\rm b)}] A compact quantum group $G$ for which the associated
Hopf $C^*$-algebra ${Q}_G$ is the inductive limit of Hopf $C^*$-algebras ${ Q}_{G_n}$ corresponding to totally disconnected quantum groups,  is itself totally disconnected.  Indeed, on one hand $\Rep(G)$ is the inductive limit of the full subcategories $\Rep(G_n)$, and, on the other hand, if $K$ is a connected quantum subgroup of $G$ then the full subcategory of $\Rep(K)$ with objects the subobjects of the restrictions of the objects of $\Rep(G_n)$ defines a connected quantum subgroup of $G_n$ so it must correspond to the trivial group since $G_n$ is totally disconnected.
 \end{itemize}
\end{exs}
In next section we shall show that the converse of Proposition \ref{sufficientfortotdisc} does not hold in general.
\begin{defn}\label{profinite}
A compact quantum group is {\em profinite} if its Hopf $C^*$-algebra is the inductive limit of finite-dimensional Hopf $C^*$-subalgebras. Equivalently, $\Rep(G)$ is the inductive limit of full, finite, tensor $C^*$-subcategories with conjugates and subobjects.
\end{defn}

 If $G$ is a profinite quantum group, all of its representations, even reducible ones, are torsion. In particular, profinite quantum groups are totally disconnected. We next show that this is in fact a characterization of profiniteness.

\begin{prop}\label{carattprofiniteness}
A compact quantum group is profinite if and only if every object of $\Rep(G)$ is a torsion object.
\end{prop}
\begin{proof}
If every object of $\Rep(G)$ is torsion, then the direct sum of any finite family of representations is a torsion object, hence the full tensor $^*$-subcategory with conjugates and subobjects generated by this family contains only finitely many irreducible representations. On the other hand,  $\Rep(G)$ is inductive limit of these finite subcategories. The last statement is a consequence of Examples \ref{examples}a, b.
\end{proof}
Every compact totally disconnected (classical) group is profinite, and indeed finite if it is a Lie group. We next see that there are totally disconnected compact matrix quantum groups which are not profinite. By Propositions \ref{sufficientfortotdisc} and \ref{carattprofiniteness}, it suffices to exhibit an example   admitting non-torsion reducible representations and such that all irreducible ones are torsion.

As already mentioned in the introduction, the main point is that there is a connection with the generalized Burnside problem in classical group theory. This problem asks whether any torsion finitely generated group is finite, and was answered in the negative by Golod and Shafarevich \cite{GS, G}. Adian and Novikov proved that the Burnside problem with bounded exponents has a negative answer as well \cite{AN}.
\begin{prop}\label{Burnside}
Let $\Gamma$ be a counterexample to the generalized Burnside problem, i.e., an infinite, finitely generated, discrete group such that every element has finite order. Then $G=C^*(\Gamma)$ is a totally disconnected compact matrix quantum group with non-torsion representations, hence it is not profinite.
\end{prop}
\begin{proof}
As irreducible representations of $C^*(\Gamma)$ correspond to group elements, they are all torsion objects, hence $G$ is totally disconnected by Proposition \ref{sufficientfortotdisc}. The Grothendieck semiring of $C^*(\Gamma)$ identifies with ${\mathbb N}\Gamma$.

If $S$ is a subset of ${\mathbb N}\Gamma$, consider the set $E(S)$ of group elements appearing in the linear combinations of the elements of $S$.
To show existence of non-torsion representations, it suffices to find an element $A$ of ${\mathbb N}\Gamma$ such that, setting $S_A:=\{A^n, n=0,1,2,\dots\}$, the associated set $E(S_A)$ is infinite. If $g_1,\dots, g_N$ is a set of generators of $\Gamma$, and $A:=g_1+\dots+g_N$, we have $E(S_A)=\Gamma$. Finally, note that $A$ is a unitary representation of $C^*(\Gamma)$ with coefficients $\{g_1,\dots, g_N\}$, hence $C^*(\Gamma)$ is a compact matrix quantum group.
\end{proof}

\begin{rem}\label{nonamenable}
Most of the counterexamples to the Burnside problem are highly non-commutative.  Olshankii constructed the first non-amenable example \cite{Olshanskii2} providing at the same time the first example to the problem of  von~Neumann of whether there exist non-amenable groups without non-abelian free subgroups. Adian proved that the free Burnside groups $B(m, n)$ are non-amenable for large odd exponents $n$ and $m>1$ \cite{Adian}. Recently, the groups of Golod and Shafarevich were shown to be non-amenable as well \cite{EJ-Z}.

However, amenability does not suffice to yield finiteness. An example of intermediate growth, hence amenable, has been constructed by  Grigorchuk \cite{G1}, thus answering negatively to  Milnor's problem of whether the growth of a group must be either polynomial or exponential.

 On the other hand, by a well known result of  Gromov \cite{Gromov}, any finitely generated group of polynomial growth is almost nilpotent. These classes of groups do have finite torsion subgroups. Therefore,  the first class of quantum groups for a positive answer    is that for which the tensor product of two representations is commutative up to equivalence. This topic will be considered more extensively in Section \ref{noetherianity}.
\end{rem}

 \medskip

\section{Normality of $\go$ and profiniteness of $\go\backslash G$}\label{normsect}

If $G$ is a compact group, the connected component of the identity $\go$ is a closed normal subgroup and one can thus form the {\it component group} $\go \backslash G$, of which it is desirable to have a non-commutative analogue. The aim of this section is to give necessary and sufficient conditions for the normality, in the sense of Wang, of the identity component of a compact quantum group.
This shall involve an analysis of the torsion subcategory of $\Rep(G)$. We give examples where $G^\circ$ is not normal.  We shall also give conditions guaranteeing that the associated quantum component group is finite or  profinite.

 \subsection{The torsion subcategory}\
\begin{defn}
Let $\T$ be a tensor $C^*$-category with conjugates, subobjects and direct sums. The full subcategory $\T^t$, whose objects are the torsion objects of
$\T$, will be called the {\em torsion subcategory} of $\T$.
\end{defn}
We have already noted that
$\T^t$ in general is not closed under tensor products.
Moreover, as we have seen from the examples related to the Burnside problem, direct sums of torsion representations of compact quantum groups may result in a non-torsion representation, and this may happen even if tensor products of irreducible objects of $\T^t$ are torsion, since in those examples every element has finite order. On the other hand, we remark that closure under direct sums is stronger than closure under tensor products.
 \begin{prop}
The torsion subcategory $\T^t$  is a $C^*$-category with conjugates and subobjects. If direct sums of irreducible torsion objects are torsion then $\T^t$ also has tensor products and direct sums.
\end{prop}
\begin{proof}
By Proposition \ref{subconj}, $\T^t$ is a $C^*$-category with conjugates and subobjects. Moreover, every torsion object is a direct sum of irreducible subobjects, which are torsion, hence finite direct sums of torsion objects are torsion. Moreover, the tensor product of two torsion objects is a subobject of the tensor square of the direct sum, hence it is torsion.
\end{proof}

The following are sufficient conditions.
\begin{prop}\label{commutativetorsion}
Assume that either
\begin{itemize}
\item[{\rm a)}]
the  objects of $\T^t$ commute up to equivalence, or
\item[{\rm b)}]
$\T^t$ has finitely many inequivalent irreducible representations and is closed under finite tensor products of them.
\end{itemize}
Then $\T^t$  is a tensor $C^*$-category with conjugates, subobjects and  direct sums.
\end{prop}

 \begin{proof}
a)
Let $u$ and $v$ be torsion objects of $\T$ and let $F_u$ and $F_v$ be the finite sets of irreducible representations appearing in the full tensor $C^*$-subcategories generated by $u$ and $v$ respectively. By the commutativity assumption, the full tensor subcategory with conjugates generated by $u v$ has, as objects, the set of all $(uv)^n(\overline{v}\,\overline{u})^m\simeq u^n\overline{u}^mv^n\overline{v}^m$, $m,n=0,1,2,\dots,$ which decompose as a direct sum of elements in $F_uF_v$, which is a finite set of possibly reducible objects, in turn decomposing into direct sums of finitely many irreducible objects. This shows that $uv$ is a torsion object. Similarly, the full tensor subcategory generated by $u\oplus v$ has, as objects, the set of all $(u\oplus v)^n(\overline{u}\oplus\overline{v})^m$ which are direct sums of objects of the form $u^r\overline{u}^{s}v^{n-r}\overline{v}^{m-s}$, whose addenda still lie in $F_u F_v$.
b) is immediate.
\end{proof}
\begin{rem}
The commutativity requirement of a) can be weakened to the requirement that $uv$ and $vu$ are  quasi-equivalent (i.e., are supported on the same set of irreducible representations)  for any pair of torsion objects $u$, $v$.
\end{rem}

\subsection{A normal sequence converging to $G^{\rm n}$}\label{torsiondegsubsect}\
\medskip

\noindent
In the remainder of this section we assume, unless otherwise stated, that all tensor subcategories of the representation category of a given compact quantum group, are full, with conjugates, subobjects and direct sums.

Let $G$ be a compact quantum group. We construct a possibly transfinite decreasing sequence of normal quantum subgroups of $G$, having $G^{\rm n}$ as a limit group. We use it to derive a characterization of normality of $G^\circ$. We also exhibit a class of examples for which $G^\circ$ is not normal.

First of all, notice that every full tensor subcategory of $\Rep(G)$ with conjugates, subobjects and direct sums is uniquely determined by the set of its objects\footnote{We are implicitly assuming, as in \cite{GLR},  that all objects belong to a given universe.}. This is a unital subsemigroup of the set of objects of $\Rep(G)$ closed under the same operations. Conversely, any subsemigroup with these properties corresponds to a full tensor subcategory of $\Rep(G)$ with the required structure. As a consequence, we can consider unions and intersections of arbitrary families of such subcategories, and the result will be normal if in addition each element of the family is normal.

We recursively define by transfinite induction a family of normal tensor subcategories ${\mathcal N}_\alpha$ of $\Rep(G)$, indexed by ordinals, and a corresponding family of subgroups $G_\alpha \subset G$, as follows. Set ${\mathcal N}_0 = \langle \iota \rangle$, $G_0 = G$; if $\beta= \alpha + 1$ is a successor ordinal, and $G_\alpha$ is defined, let ${\mathcal N}_{\beta}$ be the smallest normal tensor subcategory of $\Rep(G)$ containing all the irreducible representations $v\in\Rep(G)$ such that $v\upharpoonright_{G_{\alpha}}$ contains a torsion representation of $G_{\alpha}$. If instead $\beta$ is a limit ordinal, and $G_\alpha$ is defined for all $\alpha<\beta$, we set ${\mathcal N}_\beta:=\cup_{\alpha<\beta}{\mathcal N}_\alpha$. In both cases, we set $G_\beta$ to be the normal quantum subgroup of $G$ such that
$${\mathcal N}_\beta=\Rep(G_\beta\backslash G).$$

For instance, ${\mathcal N}_1$ is the smallest normal tensor subcategory of $\Rep(G)$ containing $\Rep(G)^t$, i.e., the class of objects of ${\mathcal N}_1$ is the intersection of the class of objects of all normal tensor subcategories of $\Rep(G)$ containing $\Rep(G)^t$. Since the torsion subcategory $\Rep(G)^t$ is not tensorial, and we have no reason to believe that it is normal, ${\mathcal N}_1$ will be in general strictly larger than $\Rep(G)^t$.

Notice that ${\mathcal N}_\alpha \subset {\mathcal N}_{\alpha+1}$, as all irreducible objects lying in ${\mathcal N}_\alpha = \Rep(G_\alpha\backslash G)$ have a trivial restriction to $G_\alpha$, and are therefore torsion in $\Rep(G_\alpha)$; similarly, ${\mathcal N}_\beta$, when $\beta$ is a limit ordinal, is by definition larger than all ${\mathcal N}_\alpha, \alpha < \beta$. We may conclude that ${\mathcal N}_\alpha \subset {\mathcal N}_\beta$, and consequently $G_\alpha \supset G_\beta$, whenever $\alpha < \beta$.

Note that $G_\delta=G_{\delta+1}$ if and only if ${\mathcal N}_\delta={\mathcal N}_{\delta+1}$, and, if this is the case, the  sequences stabilize, i.e., $G_\alpha=G_\delta$ and ${\mathcal N}_\alpha={\mathcal N}_\delta$ for $\alpha\geq \delta$.
On the other hand, ${\mathcal N}_\alpha$, and hence $G_\alpha$, must stabilize for cardinality considerations.
\begin{defn}
The smallest ordinal $\delta$ such that $G_\delta=G_{\delta+1}$ will be called the {\em (normal) torsion degree} of $G$.
\end{defn}
For example, if the Hopf $C^*$-algebra $Q_G$ is separable then the torsion degree of $G$ is a countable ordinal. The following motivating example shows that the torsion degree of a cocommutative quantum group $G=C^*(\Gamma)$ cannot exceed the first infinite ordinal, regardless of the cardinality of $\Gamma$.
\begin{ex}\label{atmostomega}
If $G=C^*(\Gamma)$, our construction reduces to the following construction of \cite{Chiodo}. Set $N_1=\langle\Gamma^t\rangle$ and let $N_r, r>1,$ be the (normal) subgroup generated by elements $g\in\Gamma$ for which $g^n\in N_{r-1}$ for some $n>0$. Then ${\mathcal N}_r$ is the normal subcategory associated to $N_r$.

It is easy to see that $\cup_{r\geq 1} N_r$ is a normal subgroup of $\Gamma$ contained in the torsion-free radical $\Lambda^\circ$ as defined in \cite{BH}; also see Example \ref{exconnectedcomponent} above. Moreover, it is easy to check that $\cup N_r\backslash\Gamma$ is torsion free, hence $\cup N_r=\Lambda^\circ$, so $\cup_r{\mathcal N}_r=\Rep(\go\backslash G)$. Hence $\go= G_\omega$ is determined by the limit of the sequence.

If we apply the construction of $\Lambda^\circ$ to  $\Lambda^\circ$ itself, we obtain $\Lambda^\circ$ again, since $\Gamma$ and $\Lambda^\circ$ have the same torsion part. Hence  $\go\backslash G$ is totally disconnected also in the  case where $\Gamma^t$ is not a subgroup.
\end{ex}

Our next aim is to determine the limit of the sequence $G_\alpha$  in the general case.
\begin{lemma}\label{lemma} For each ordinal $\alpha$, ${\mathcal N}_\alpha\subset \Rep(G^{\rm n}\backslash G)$, hence $G^{\rm n}\subset G_\alpha$. \end{lemma}
\begin{proof}
Every torsion representation of $G$ is trivial on $\go$, and hence also on $G^{\rm n}$.
We therefore have
an inclusion of full subcategories of $\Rep(G)$,
$$\Rep(G)^t\subset\Rep(G^{\rm n}\backslash G).$$
Hence
$$\Rep(G_1\backslash G)={\mathcal N}_1\subset\Rep(G^{\rm n}\backslash G),$$
since $ \Rep(G^{\rm n}\backslash G)$ is normal. As before, we obtain
 $G^{\rm n}\subset G_1.$
An inclusion $G^{\rm n}\subset G_\alpha$ implies that if $v$ is an irreducible   of $G$ such that $v\upharpoonright_{G_\alpha}$ contains a torsion representation of $G_\alpha$ then $v\upharpoonright_{G^{\rm n}}$ has invariant vectors since $G^{\rm n}$ is connected.
But then $v\in \Rep(G^{\rm n}\backslash G)$ by normality of $G^{\rm n}$.
Hence, ${\mathcal N}_{\alpha+1}\subset\Rep(G^{\rm n}\backslash G)$ implying $G^{\rm n}\subset G_{\alpha+1}$. If $\alpha$ is a limit ordinal and if $G^{\rm n}\subset G_\beta$ for $\beta<\alpha$, then
${\mathcal N}_\alpha=\cup_{\beta<\alpha}{\mathcal N}_\beta\subset\Rep(G^{\rm n}\backslash G),$ hence $G^{\rm n}\subset G_\alpha.$
\end{proof}
\begin{thm}\label{ntd} Let $G$ be a compact quantum group. The torsion degree of $G$   is the smallest ordinal  $\delta$ such that $G_\delta$ is connected. Moreover, $G_\delta=G^{\rm n}$ and ${\mathcal N}_\delta=\Rep(G^{\rm n}\backslash G)$. In particular,
${\mathcal N}_\delta=\Rep(G)$ if and only if
$G^{\rm n}$ is the trivial group.
\end{thm}
\begin{proof}
We know that $G^{\rm n}\subset G_\alpha$ for all ordinals $\alpha$. If some $G_\alpha$ is connected then $G_\alpha=G^{\rm n}$ since $G^{\rm n}$ contains every normal connected quantum subgroup of $G$. Therefore the  sequences stabilize for $\beta\geq \alpha$.

Conversely, let us assume that $G$ has torsion degree $\delta$. We show that $G_\delta$ is connected. Let $v$ be an irreducible of $G$ such that $v\upharpoonright_{G_\delta}$ contains a torsion representation of $G_\delta$. Then $v\in{\mathcal N}_{\delta+1}={\mathcal N}_\delta=\Rep(G_\delta\backslash G)$. Hence, $v\upharpoonright_{G_\delta}$ is a multiple of the trivial representation. This shows that $\Rep(G_\delta)$ is torsion free, i.e., $G_\delta$ is connected, or equivalently $G_\delta \subset G^{\rm n}$. The remaining statements follow easily.
\end{proof}

\begin{cor} If $\go$ is normal and if $G$ has torsion degree $\delta$ then $G_\delta=\go$ and ${\mathcal N}_\delta=\Rep(\go\backslash G)$.\end{cor}

\begin{cor}\label{necsuff}
Let $G$ be a compact quantum group. Then $G^\circ$ is normal if and only if for every ordinal $\alpha$, every representation of ${\mathcal N}_\alpha$ restricts to a multiple of the trivial representation of $G^\circ$.
\end{cor}
\begin{proof}
If the statement holds for all ordinals, it certainly holds for the torsion degree $\delta$. Therefore, all representations of ${\mathcal N}_\delta$ restrict to some multiple of the trivial representation of $G^\circ$. We argue that there are more $G^\circ$-invariant vectors than $G^{\rm n}$-invariant vectors in the irreducible representations of $G$, by normality of $G^{\rm n}$. It follows that $G^\circ\subset G^{\rm n}$, hence equality holds.
The converse follows from  Lemma \ref{lemma}.
\end{proof}

The next proposition determines the sequences   associated to $G^{\rm n}\backslash G$. We shall identify $\Rep(G^{\rm n}\backslash G)$ with a full tensor subcategory of $\Rep(G)$ in the natural way. We omit the proof.

\begin{prop}
The quantum groups  $G$ and $G^{\rm n}\backslash G$ have the same associated sequences
of normal subcategories of $\Rep(G)$ and hence the same torsion degree. Moreover,
for every ordinal $\alpha$,  $(G^{\rm n}\backslash G)_\alpha=G^{\rm n}\backslash G_\alpha$. In particular,  $(G^{\rm n}\backslash G)^{\rm n}$ is the trivial group.
\end{prop}

\begin{rem}
The problem (cf. Remark \ref{totdisc3}) of whether normality of $\go$ implies that $\go\backslash G$ is totally disconnected in general, is equivalent to that of normality of the identity component of $\go\backslash G$, by the previous proposition.
\end{rem}

The following result gives some information on the ordinals that can possibly arise as values of the torsion degree of a compact quantum group. It is a  generalization of the cocommutative case.

\begin{prop}\label{tdcocommutativesubgroup}
Let $G$ be a compact quantum group. If $G_\alpha$ is cocommutative for some ordinal $\alpha$, then so is $G_\beta$ for every $\beta>\alpha$, and the torsion degree of G is $\leq \alpha + \omega$. In particular, $G^{\rm n}$ is cocommutative.
\end{prop}
\begin{proof}
The second assertion follows from the first and the fact that $G_\beta$ is a
quantum subgroup of $G_\alpha$ for $\beta\geq\alpha$.
If the restriction of some irreducible representation $v$ of $G$ to $G_{\alpha+\omega}$ contains a torsion representation, then the same holds for some irreducible component,
say $g$, of $v\upharpoonright_{G_\alpha}$. The sequence $G_\alpha\supset G_{\alpha+1}\supset\dots \supset G_{\alpha+\omega}$ is described by an associated sequence
$M_0\subset M_1\subset\dots \subset M_{\omega}$ of normal subgroups of the dual of $G_\alpha$ such that $M_\omega=\cup_n M_n$.
Since some power of $g$
lies $M_{\omega}$, it must lie in $M_{n}$, for some $n<\omega$. This shows that $v\in{\mathcal N}_{\alpha+n}\subseteq {\mathcal N}_{\alpha+\omega}=
\Rep(G_{\alpha+\omega}\backslash G)$, hence $v\upharpoonright_{G_{\alpha+\omega}}$ is a multiple of the trivial representation, and therefore $G_{\alpha+\omega}$ is connected.
\end{proof}

\subsection{Examples }\label{variousexamples}\
\medskip

\noindent In this subsection we give examples of compact quantum groups with  torsion degrees $\leq\omega$ or with non-normal identity component.

The quantum groups with   profinite quotient $G^{\rm n}\backslash G$   have torsion degree $\leq1$.    In particular, this includes the profinite quantum groups, by Proposition \ref{carattprofiniteness}.
They  have the simplest  sequences, ${\mathcal N}_r=\Rep(G)^t$
for all $r\geq1$.
\begin{prop}\label{tdprofiniteness} Let $G$ be a compact quantum group.
Then $\Rep(G)^t=\Rep(G^{\rm n}\backslash G)$ if and only if  $G^{\rm n}\backslash G$ is profinite. In this case,
\begin{itemize}
\item[{\rm a)}]
$\Rep(G)^t$ has tensor products and direct sums, moreover it is the inductive limit of full finite tensor $^*$-subcategories of $\Rep(G)$ with conjugates and subobjects,
\item[{\rm b)}]
$\Rep(G)^t$ is a normal subcategory of $\Rep(G)$.
\end{itemize}
\end{prop}
\begin{proof}
The stated characterization follows easily from the equality $\Rep(G)^t=\Rep(G^{\rm n}\backslash G)^t$ and from Proposition \ref{carattprofiniteness}.
 \end{proof}

Note that the case of torsion degree $\leq1$ also includes, by Propositions
\ref{sufficientfortotdisc} and \ref{Burnside}, the examples  from the Burnside problem recalled in Subsection \ref{totdiscqg}. We next describe examples of higher torsion degrees.
\begin{ex}\label{exChiodo}
In \cite{Chiodo} Chiodo and Vyas exhibited an  example of a discrete group $\Gamma$ dual to a cocommutative compact quantum group of torsion degree $2$.
Their methods can be generalized to show that the same holds for any free product $\Gamma$ of $(\Cyc_m*\Cyc_n)$ and $\Cyc_\infty$, with amalgamation $xy=z^p$, where $m, n$ and $p\geq 2$ are integers, and $x, y, z$ denote generators, respectively, of the three cyclic groups. Indeed, the same arguments show that $N_1$ is the normal subgroup generated by $x$ and $y$. Moreover, it is easy to see that   the quotient $N_1\backslash \Gamma$ is $\Cyc_p$.
Hence, for all $\gamma\in \Gamma$,   $\gamma^p\in N_1$, implying $N_2=\Gamma=N_r$ for all $r \geq 2$. Therefore
$$\Rep(G)^t\subset{\mathcal N}_1\subset{\mathcal N}_2=\Rep(G).$$
Both inclusions are strict. In particular,  $\go$ is the trivial group, i.e. $G$ is totally disconnected. Note that  not all irreducible representations of $G$ are torsion.
\end{ex}

We sketch, now, how to generalize Example \ref{exChiodo} to yield cocommutative compact quantum groups with torsion degree equal to any given $\alpha \leq \omega$.

\begin{ex}\label{arbitrary}
Declare a {\em pointed group} to be a pair $P = (\Gamma, g)$ where $\Gamma$ is a group and
$g \in \Gamma$ is a distinguished element. If $P = (\Gamma, g)$ and $P' = (\Gamma', g')$ are pointed groups, define their free product to be $P * P' = (\Gamma * \Gamma', g g')$; if $n \in \Nset$, and $H$ denotes the free product of $\Gamma$ and $\Gamma'$ with amalgamation\footnote{Here we are implicitly assuming that $g, g'^n$ have the same order.} $g = g'^n$, then we also set $P \circ_n P' = (H, [g'])$.

Let us now recursively define a sequence of pointed groups $P_k = (\Gamma_k, \gamma_k), k \in \Nset,$ by choosing $P_1 = (\{\pm 1\}, -1)$, and setting $P_{k+1} = (P_k * P_k) \circ_2 \Cyc_\infty$. For instance, $\Gamma_1 \simeq \Cyc_2$, and $\Gamma_2$ is as in Example \ref{exChiodo}, when $l = m = n = 2$. Torsion elements in a free product with amalgamation are those obtained by conjugating torsion elements of the free factors; it is then clear that the quotient of $\Gamma_{k+1}$ by the normal subgroup generated by its torsion elements is isomorphic to $\Gamma_k$. An easy induction shows that the compact quantum group $C^*(\Gamma_k)$ has torsion degree equal to $k$, for all $k \in \Nset$.

Notice now that each $\Gamma_k$ injects into $\Gamma_{k+1}$ by mapping to the first free factor. The limit group $\Gamma_\omega$ of this direct system is isomorphic to its quotient by the normal subgroup generated by its torsion: this prevents the torsion degree of $C^*(\Gamma_\omega)$ from being finite. However, as observed in Example \ref{atmostomega}, every cocommutative compact quantum group has torsion degree $\leq \omega$.
\end{ex}

We next discuss non-normality of the identity component. Free products between compact quantum groups have been introduced by Wang in \cite{Wang_free}.

\begin{thm}\label{non-normal}
Let $G^\circ$ be a connected compact quantum group and let $\Gamma$ be a discrete  group   and consider the free product quantum group $G=G^\circ*C^*(\Gamma)$.
\begin{itemize}
\item[{\rm a)}] If $G^\circ$   has a
non-trivial irreducible representation of dimension $>1$ and $\Gamma$ has a non-trivial element of finite order then the identity component of $G$ contains $G^\circ$ as a subgroup and  is not normal.
Irreducible objects of $\Rep(G)^t$ correspond to elements of $\Gamma^t$.
\item[{\rm b)}]
If $\Gamma$ is a torsion group  then $G^\circ$ is the identity component of $G$.
\item[{\rm c)}] If $G^\circ$ has no non-trivial irreducible representation of dimension $1$ and if $\Gamma$ is a torsion group then  $G$ has torsion degree $\leq\omega$; the subgroups $G_k$  of $G$ are central quantum  subgroups of  $G^\circ$ for all $1\leq k\leq\omega$.
  \end{itemize}
\end{thm}
\begin{proof}
a)
By the universality property of free product quantum groups,    $G^\circ$ is a quantum subgroup of $G$,   connected by assumption, hence a subgroup of the identity component of $G$.
Let now   $u\in\Rep(G)$ and $\gamma\in\Gamma$ be as required.    Then $\overline{u}\gamma u$ is a non-trivial irreducible representation of $G$, by \cite[Theorem 3.10]{Wang_free}.
Hence, if  ${\mathcal S}$ is any
tensor subcategory of $\Rep(G)$ then either $(\overline{u}\gamma u)_{\mathcal S}=\overline{u}\gamma u$ or $ (\overline{u}\gamma u)_{\mathcal S}=0$.
If in addition ${\mathcal S}$ is normal and contains $\Rep(G)^t$ then the former holds,
by Proposition \ref{normalcat} a. Hence  $\overline{u}\gamma u$ is an object of  ${\mathcal S}$, and therefore of  ${\mathcal N}_1$. But the restriction of $\overline{u}\gamma u$
to $G^\circ$ is $\overline{u}u$, since $\gamma$ is a torsion one-dimensional representation of $G$ and $G^\circ$ is connected. By Frobenius reciprocity and Peter-Weyl theory, this is not a multiple of the trivial representation since $u$ is irreducible and of dimension $>1$, thus contradicting Corollary \ref{necsuff}. Hence the identity component is not normal.
Note that the irreducible torsion representations of $G$ arise precisely  from the torsion elements of $\Gamma$.

b) Since $G^\circ$ is a connected quantum subgroup
of $G$,
we need to verify that any connected quantum subgroup $K$ of $G$ is in fact a subgroup of $G^\circ$.
Now,  $G^\circ$ is also a quotient quantum group of $G$, hence the defining epimorphism $\pi: {\mathcal Q}_G\to{\mathcal Q}_K$   can be restricted to the Hopf subalgebra
$\pi^\circ: {\mathcal Q}_{G^\circ}\to{\mathcal Q}_K$.
By connectedness of $K$, $\pi$ acts trivially on $\Gamma$, hence  $\pi^\circ$ is an epimorphism as well. Since $\pi$ factors through ${\mathcal Q}_G\to{\mathcal Q}_{G^\circ}\to{\mathcal Q}_K$,    $K$ is in fact a subgroup of $G^\circ$.

c)  We may assume $G^\circ\neq1$. Recall that $G_1$ was defined by ${\mathcal N}_1=\Rep(G_1\backslash G)$ and that ${\mathcal N}_1$ contains the torsion subcategory
of $G$. Therefore  torsion representations of $G$ are trivial on $G_1$. In particular, this holds for the elements of $\Gamma$.
We show that $G_1$ is in fact a quantum subgroup of $G^\circ$. To this aim,
 we may argue just as in the proof of b) with $G_1$ in place of $K$, and replace the argument involving connectedness of $K$ with  the previous observation.
We next see that $G_1$ is central.
If $u$ and $\gamma$ are as in the proof of a), ${\overline u}\gamma u\in{\mathcal N}_1$. Since ${\overline u}\gamma u$
has the same restriction as $\overline{u}u$ to $G_1$, we see that
$\overline{u}u$ is trivial on   $G_1$, as well. Hence
$u\upharpoonright_{G_1}$ is a multiple of  a one-dimensional representation. By the arbitrary choice of $u$, $G_1$ is then central. Hence $G$ has torsion degree $\leq\omega$ and moreover
$G^{\rm n}$, as well as all the $G_k$, are cocommutative, by
Proposition \ref{tdcocommutativesubgroup}.

Finally, we note that if we have an intermediate quantum subgroup $N\subset K\subset G$ with $N$ normal in $G$ then $N$ must also be normal in $K$.
Indeed, every irreducible representation of $K$ is a subrepresentation of the restriction of some irreducible representation of $G$.
Hence $G_k$ is normal in $G^\circ$ for $1\leq k\leq\omega$.
\end{proof}

\begin{rem}
Note that, if in a)  $\Gamma^t$ is in addition a subgroup, $\Rep(G)^t$ is tensorial but not normal.
\end{rem}

The previous theorem shows that strong restrictions on the value of the torsion degree of a free product $G^\circ *C^*(\Gamma)$ may be due to shortness of central quantum subgroups  of $G^\circ$.
In particular, we note the following result.

\begin{cor}\label{corolnon-normal} Let $G$ be a compact quantum group of the form  $G=G^\circ*C^*(\Gamma)$, where $G^\circ$ is a semi-simple compact Lie group and $\Gamma$ is a non-trivial torsion group. Then $G^\circ$ is the identity component of $G$,  $G^{\rm n}$ is the trivial group, and $G$ has torsion degree $\leq2$. In particular, when $G^\circ$ is adjoint, then $G$ has torsion degree $1$.
\end{cor}
\begin{proof}
Central subgroups of $G^\circ$ are finite. In particular, by the previous theorem $G_1$ is finite, hence ${\mathcal N}_2=\Rep(G)$ and the conclusions easily follow.
\end{proof}

We next give an example showing that torsion degree $2$ can arise from the examples  considered in the previous corollary.

\begin{ex}
Consider $G={\rm SU}(2)*C^*(\Gamma)$, where $\Gamma$ is any non-trivial torsion group, and denote by $u_n$ be the self-conjugate irreducible representation of ${\rm SU}(2)$ of dimension $n+1$. The set of irreducible objects  of the category $\Rep({\rm SU}(2))$ has a unique $\Zset_2$-grading making $u_n$ even or odd according to the parity of $n$. We may then extend this grading to all the irreducible objects of $\Rep(G)$ by setting every $\gamma \in \Gamma$ to be even.

Let now ${\mathcal N}$ be the full subcategory of $\Rep(G)$ whose objects are finite direct sums of even irreducible objects. Clearly, ${\mathcal N}$ is a tensor subcategory of $\Rep(G)$; it contains $\Rep(G)^t$ by Theorem \ref{non-normal}a; finally, it satisfies the requirements of Proposition \ref{sufficientfornormality}d, and is therefore normal and even central by Proposition \ref{centrality}. We conclude that
${\mathcal N}_1\subset{\mathcal N}\subset \Rep(G)={\mathcal N}_2$, and the latter inclusion is strict as $u_1\notin{\mathcal N}$.
\end{ex}

\begin{ex}
A combination of Theorem \ref{non-normal} and Example \ref{exChiodo} can be used to give a compact quantum group $G$ for which the tensor category ${\mathcal T}_1$ generated by the torsion subcategory $\Rep(G)^t$ is normal but $G$ has torsion degree $>1$ and a non-normal the identity component. In particular, ${\mathcal T}_1$ has infinitely many irreducible representations.
\end{ex}

\subsection{The problem of approximating $G^\circ$}\
\medskip

We propose a variant of the methods of the previous subsection, yielding a second transfinite decreasing sequence of quantum subgroups of $G$ --- denoted by
$K_\alpha$, where $\alpha$ is an ordinal --- which is more likely to approximate $G^\circ$ if normality of the latter is not known.

Define recursively ${\mathcal M}_0:=\langle\iota\rangle$, $G_0=G$; if $\beta = \alpha+1$ is a successor ordinal and $K_\alpha\subset G$ is defined, set $\M_\beta$ to be the smallest normal tensor subcategory of $\Rep(K_\alpha)$ containing $\Rep(K_\alpha)^t$, and let $K_{\beta}$ be the normal quantum subgroup of $K_\alpha$ such that
$${\mathcal M}_{\beta}=\Rep(K_{\beta}\backslash K_\alpha).$$
In particular, $K_1=G_1$.
If $\beta$ is a limit ordinal, and $K_\alpha\subset G$ is defined for $\alpha<\beta$,
let $K_\beta$ be the compact quantum group with representation category
determined by
$$(u\upharpoonright_{K_\beta}, v\upharpoonright_{K_\beta})=\cup_{\alpha<\beta}
(u\upharpoonright_{K_\alpha}, v\upharpoonright_{K_\alpha}),$$
for $u$, $v\in\Rep(G)$.
By construction, for $\alpha<\beta$, the inclusion $\Rep(K_\alpha)\subset\Rep(K_\beta)$
is compatible with the embedding into the Hilbert spaces, hence $K_\alpha\supset K_\beta$. We see that $K_\delta=K_{\delta+1}$
 if and only if $K_\delta$ is connected, and then $K_\alpha= K_\delta$ for all $\alpha>\delta$.

\begin{prop}
Any decreasing sequence   of quantum subgroups of $G$ indexed by the ordinals stabilizes.
\end{prop}
\begin{proof}
Let $J$ be a set of the same cardinality, that we may assume infinite, as that of a complete set $(u_j)$ of inequivalent irreducible representations of $G$. By Tannaka--Krein duality and Frobenius  reciprocity, a quantum subgroup $K$ of $G$ is uniquely determined by the specification of the subspace  of
$K$--invariant vectors in the Hilbert space $H_j$ of $u_j$, which we regard as a   vector subspace of $V:=\Pi_{j\in J}H_j$. The given sequence   therefore corresponds to an increasing sequence of subspaces of $V$, which stabilizes
 if the cardinality of the corresponding ordinal  strictly exceeds
that of $J$.
\end{proof}

\begin{thm}\label{suffnormality}
Assume that for each ordinal $\alpha$, the smallest tensor subcategory ${\mathcal T}_{\alpha+1} \subset \Rep(K_\alpha)$ containing $\Rep(K_\alpha)^t$ is already normal. Then for every ordinal, $G^\circ\subset K_\alpha$ and $K_\alpha$ stabilizes to $G^\circ$.  If in addition $G$ has torsion degree $\leq1$,  then   $\go$  is normal, $\Rep(\go\backslash G)={\mathcal T}_1$ and $\go\backslash G$ is totally disconnected.
\end{thm}
\begin{proof}
Since every torsion representation of $G$ is trivial on $\go$ by Proposition \ref{restrictingtorsiontoconnectedcomponent}, so is every representation of ${\mathcal T}_1=\Rep(K_1\backslash G)$. On the other hand, an irreducible representation of $\Rep(K_1\backslash G)$ is precisely an irreducible representation of $G$ restricting to a multiple of the trivial representation on $K_1$.
Taking into account Proposition \ref{normal}c for $K_1$, $\go$ has more invariant vectors than $K_1$ in the spaces of irreducible representations of $G$, so by \eqref{3.2},  $\Rep(K_1)\subset\Rep(\go)$, hence $\go\subset K_1$.

On the other hand, if $K$ is an intermediate quantum subgroup of
$G$, $G^\circ\subset K\subset G$, then $K^\circ=G^\circ$. Hence, applying the first part
to an inclusion $G^\circ\subset K_\alpha$ gives $G^\circ\subset K_{\alpha+1}$.
Let $\alpha$ be a limit ordinal and assume that $G^\circ\subset K_\beta$ for $\beta<\alpha$. Then
for every irreducible representation $u$ of $G$, the space of vectors in $H_u$ invariant
under $K_\alpha$ coincides with that of vectors
invariant under all the restrictions to $K_\beta$. These are also invariant under the restriction to
$G^\circ$, hence $G^\circ\subset K_\alpha$. We thus see that $G^\circ$ is a subgroup
of the limit of the sequence. On the other hand, this limit group is connected, hence it coincides with $G^\circ$.

If $G$ has torsion degree $\leq1$ then  $K_1=\go$ and it is normal. Furthermore, $\go\backslash G$ has the same torsion subcategory as $G$, and the same category ${\mathcal T}_1=\Rep(\go\backslash G)$, obviously normal in $\Rep(\go\backslash G)$, and with torsion degree $\leq1$. By the first part of the statement applied  to $\go\backslash G$, the identity component of $\go\backslash G$ is normal, and by Theorem \ref{ntd}, $\go\backslash G$ is totally disconnected.
\end{proof}
A slight variation of the proof of the previous theorem also shows the following result.
\begin{thm}\label{variation}
Let $G$ be a compact quantum group such that $G^\circ$ is normal. Then the decreasing sequence $K_\alpha$ stabilizes to $G^\circ$.
\end{thm}

\subsection{Normality of $\go$ and profiniteness of $\go\backslash G$}\
\medskip

\noindent
We now proceed to give a characterization, motivated by the theory of Lie groups, of compact quantum groups with normal
$G^\circ$ and finite $G^\circ\backslash G$. This includes the problem of showing that   the   torsion degree is $\leq1$. The examples  discussed in Subsection \ref{variousexamples} show that the properties involved in our characterization are independent.
The proof relies on the induction theory for tensor $C^*$-categories developed in \cite{PR}.

Let us first assume that $\Rep(G)^t$ is a tensor subcategory with direct sums. We may apply the construction of Section \ref{normalsubsect} and associate a $^*$-functor
$$t:\Rep(G)\to\Rep(G)^t$$
mapping the representation $u$ to its maximal torsion subrepresentation $u_t$. The complementary subrepresentation will be denoted $u_f$, and referred to as
the free part of $u$, and one may decompose $uv$ in torsion and free part, as done in Section \ref{normalsubsect}. We shall call $u$ torsion or free if $u=u_t$ or $u=u_f$ respectively.
\begin{thm}\label{connectedcomponent}
Let $G$ be a compact quantum group. Then the following are equivalent,
\begin{itemize}
\item[{\rm a)}] $G^\circ$ is normal in $G$ and $G^\circ\backslash G$ is finite,
\item[{\rm b)}] $\Rep(G)^t$ is tensorial, finite and normal.
\end{itemize}
In this case,
 $G$ has torsion degree $\leq1$. Moreover,
\begin{itemize}
\item[{\rm c)}]
$\Rep(\go\backslash G)= \Rep(G)^t$,
\item[{\rm d)}]
$\Rep(\go)$ is determined by
$$(u\upharpoonright_{\go}, v\upharpoonright_{\go})=\{\overline{R}^*\otimes 1_v\circ 1_u\otimes\phi, \quad \phi\in H_{({\overline{u} v})_t}\},$$
where $u$, $v$ are irreducible representations of $G$ and $\overline{R}\in(\iota, u\overline{u})$ is non-zero.
\end{itemize}
\end{thm}
\begin{proof} a)$\Rightarrow$ b) follows from Proposition   \ref{tdprofiniteness}. b) $\Rightarrow$ a)
Since $\Rep(G)^t$ is a normal tensor  subcategory
of $\Rep(G)$, it is the representation category of $G_1\backslash G$. Moreover,  properties   c) and d) in the statement hold with $G_1$ in place of $\go$, by Theorem \ref{carattnormal}.
By  Theorem \ref{suffnormality}, $\go\subset G_1$.
 We are left to show that $G_1$ is connected. To this aim, note that by d) applied to $G_1$, for every irreducible torsion representation $u$ of $G$, the arrow space $(\iota, u\upharpoonright_{G_1})$ has full dimension, hence $u\upharpoonright_{G_1}$ is a multiple of the trivial representation. We are left to show that every irreducible free representation of $G$ after restriction to $G_1$ is still free.
We shall apply the theory of induction (and use notation) of \cite{PRinduction} to the tensor categories $\A=\Rep(G)$, $\M=\Rep(G_1)$, the
embedding functor $\tau=H$ of $\Rep(G)$ and the restriction functor $\mu:v\in\Rep(G)\to v\upharpoonright_{G_1}\in\Rep(G_1)$.

By \cite[Theorem 6.2]{PRinduction}, for each representation $u$ of $G$ there is a Hilbert bimodule
representation $\text{Ind}(\mu_u)$ of $G$ on a canonical Hilbert $G$-bimodule
$\H_u$ over the  coefficient $C^*$-algebra $\C=Q_{G_1\backslash G}$. By \cite[Theorem 6.4]{PRinduction}, the functor $\mu_u\to\text{Ind}(\mu_u)$ is faithful, tensorial and full. Hence it suffices to show that $\text{Ind}(\mu_u)$ is a free object if $u$ is a free irreducible representation of $G$.
If $v$ is an irreducible representation of $G$, the space of the $v$-isotypic component of the
$G$--bimodule $\H_u$ is   $(\mu_v, \mu_u)\otimes H_v$, with $G$ acting trivially on   $(\mu_v, \mu_u)$. Hence no torsion  $v\in\hat{G}$ can be spectral since otherwise $0\neq (\mu_v, \mu_u)=(\dim(v)\iota, \mu_u)$ and this would imply
$\mu_u=\dim(u)\iota$ since $G_1$ is normal, which in turn would imply
$\dim H_u=\dim(\iota, u\upharpoonright_{G_1})=\dim H_{u_t}$ and hence
$u=u_t$.

Let  $X$ be a non-zero torsion Hilbert $G$-submodule of $\H_u$. By \cite[Theorem 6.3]{PRinduction}, $\H_u$ is linearly isomorphic to $H_u\otimes Q_{{G_1}\backslash G}$ and ${G_1}\backslash G$ is finite, hence $\H_u$, and therefore also $X$, is a finite-dimensional vector space. It follows that the set $\hat{X}$ of inequivalent irreducible spectral representations  of $G$ arising from the full tensor $^*$-subcategory with conjugates and subobjects of the category of Hilbert $G$-bimodules generated by $X$ is finite. From the formula of tensor products and conjugates of the $\H_u$, according to \cite[Sections 7.3 and 7.4]{PRinduction}, we see that $\hat{X}$ must contain all the irreducible representations of $G$ lying in the tensor category with conjugates generated by the spectral representations of $X$. Therefore all spectral representations of $X$ are torsion, but this is impossible since $\H_u$ has no torsion spectral representation of $G$, and the proof is complete.
\end{proof}
\begin{cor}\label{corollary}
Let $G$ be a compact quantum group such that
${\mathcal Q}_G=\varinjlim{\mathcal Q}_{L_n}$, where for each $n\in{\mathbb N}$, $L_n$ is a quotient quantum group
 with the property that $\Rep(L_n)^t$ is tensorial, finite and normal
in $\Rep(L_n)$. Then $\go$ is normal and $\go\backslash G$ is profinite. In particular, $G$ has torsion degree $\leq1$. Moreover,
$${\mathcal Q}_{\go}=\varinjlim{\mathcal Q}_{(L_n)^\circ},$$
$${\mathcal Q}_{\go\backslash G} =\varinjlim{\mathcal Q}_{(L_n)^\circ\backslash L_n}.$$
\end{cor}
\begin{proof} By the previous theorem, $(L_n)^\circ$ is normal in $L_n$,  $(L_n)^\circ\backslash L_n$ is finite and $\Rep(L_n)^t=\Rep((L_n)^\circ\backslash L_n)$, for all $n$.
Since $\Rep(G)=\cup_n\Rep(L_n)$ as full tensor subcategories then
$\Rep(G)^t=\cup_n\Rep(L_n)^t$. In particular, $\Rep(G)^t$
is a tensor subcategory of $\Rep(G)$ with direct sums. Normality of $\Rep(L_n)^t$
in $\Rep(L_n)$ for all $n$ implies normality of $\Rep(G)^t$ in $\Rep(G)$.
Hence $G_1\supset \go$ by Theorem \ref{suffnormality}. Moreover,
$$\Rep(G)^t=\Rep(G_1\backslash G)=\cup_n\Rep((L_n)^\circ\backslash L_n).\eqno(5.1)$$
The formula of intertwiners of $\Rep((L_n)^\circ)$ between restrictions of representations $u, v\in L_n$ to $(L_n)^\circ$ shows that the intertwiners do not change if we regard $u, v$ as objects of $L_{n+1}$ and we restrict them to $(L_{n+1})^\circ$. Therefore there is a natural inclusion
of full subcategories $\Rep((L_n)^\circ)\subset \Rep((L_{n+1})^\circ)$. Since
$\Rep(G_1)$ is determined by $\Rep(G)^t$ through a similar formula, $(5.1)$
implies
$$\Rep(G_1)=\cup_n\Rep((L_n)^\circ).$$
In particular, $\Rep(G_1)$ is torsion free, hence $G_1$ is connected implying in turn $G_1=\go$.
\end{proof}

\section{Noetherianity and finiteness of representation rings}\label{noetherianity}

  The aim of this section is to formulate properties of a geometric nature on compact quantum groups that ensure
  an analogue of the classical property that quotients of  Lie groups  by closed normal subgroups are Lie groups. More precisely, we aim to
  restrict the class of compact matrix quantum groups to a subclass which is closed under the passage to quotient quantum groups.

  In what follows, $R= R(G)$ will be the Grothendieck ring of (finite-dimensional) representations of a compact quantum group $G$.   We start observing that quotient
  quantum groups of $G$   are in one-to-one correspondence with certain subrings of $R$, and we next turn our attention to them.

\begin{defn}
A unital subring $A \subset R$ is a {\em sub-representation ring}, denoted $A < R(G)$, if it is closed under taking duals and subobjects; in other words, $A < R$ if and only if its elements are precisely the $\Zset$-linear combinations of the irreducible elements of $R$ contained in $A$.
\end{defn}
\begin{rem}
Let $\Irr(R)$ denote the set of all irreducible elements of the representation ring $R$. If $A < R$, then $\Irr(A) \subset \Irr(R)$. Sub-representation rings of $R$ are in one-to-one correspondence with full tensor subcategories (with conjugates,   subobjects and direct sums) of the category $\Rep(G)$, and are uniquely determined by their set of irreducible objects.
\end{rem}
If $X \subset \Irr(R)$, we denote by $\langle X \rangle$ the intersection of all sub-representation rings of $R$ containing $X$; this is again a sub-representation ring of $R$.
\begin{defn}
Let $G$ be a compact quantum group, $R$ its Grothendieck ring of representations. Then:
\begin{itemize}
\item
$R$ is {\em finitely generated} if it is finitely generated as a $\Zset$-algebra.
\item
$R$ is {\em Noetherian} if it is a Noetherian ring.
\item
$R$ is of {\em Lie type} if all increasing sequences
$$A_1 < A_2 < \dots < A_n \dots$$
of sub-representation rings stabilize.
\item
$R$  {\em has a generating representation} if there exists a finite subset $X \subset \Irr(R)$ such that $X \subset A < R$ implies $A = R$.
\end{itemize}
We will say that $G$ is of Lie type whenever $R$ is of Lie type. Clearly, $G$ is a compact matrix quantum group if and only if $R$ has a generating representation.
\end{defn}

\begin{rem}\label{rightleft}\qquad
\begin{itemize}
\item
The ring $R$ is endowed with an antiautomorphism which associates with every representation $v$ its conjugate representation $\overline v$. In particular, $R$ is isomorphic to its opposite ring $R^{\op}$. As a consequence, $R$ is left Noetherian if and only if it is right Noetherian.
\item
If $R$ is of Lie type and $A < R$, then $A$ is also trivially of Lie type.
\item
The Grothendieck ring $R(G)$ of a compact quantum group $G$ certainly contains strictly more information than its ring structure. Indeed, representation rings of the classical Lie groups $\SU(2)$ and $\SO(3)$ are both ring-isomorphic to the ring $\Zset[u]$; however, they are not isomorphic as representation rings, as the former contains a non-trivial sub-representation ring, whereas the latter does not.
\end{itemize}
\end{rem}

\begin{prop}\label{alsoright}
Let $A < R$. Then the injection $A \hookrightarrow R$ splits as a homomorphism of $A$-bimodules.
\end{prop}
\begin{proof}
Let $U \subset R$ be the $\Zset$-submodule generated by $\Irr(R) \setminus \Irr(A)$. Then $R$ decomposes into $A \oplus U$ as a $\Zset$-module, and $U$ is an $A$-bisubmodule of $B$ by Lemma \ref{splitinj}.
 \end{proof}

\begin{thm}\label{noetherian}
Let $A < R$. If $R$ is Noetherian, then $A$ is also Noetherian.
\end{thm}
\begin{proof}
Let $I\subset A$ be a left ideal. Then $RI$ is a left ideal of $R$. If $R = A \oplus U$ is a direct sum decomposition of (left) $A$-modules, then $RI = AI \oplus UI$. Now, $AI = I\subset A$ as $I$ is a left $A$-module and $1 \in A$; moreover $UI \subset U$ as, by Proposition \ref{alsoright}, $U$ is a right $A$-submodule. This implies that $RI \cap A = I$, hence if $I \subsetneq I'$ are left ideals of $A$, then $RI \subsetneq RI'$.

Say $A$ is not left Noetherian. Then there exists an infinite ascending sequence of proper inclusions
$$I_1 \subset I_2 \subset \dots \subset I_n \subset \dots$$
among left ideals of $A$. This yields an infinite ascending sequence of proper inclusions
$$R I_1 \subset R I_2 \subset \dots \subset R I_n \subset \dots$$
of left ideals of $R$. Noetherianity of $R$ leads now to a contradiction.
\end{proof}
\begin{rem}
Notice that both summands in the above decomposition $RI = I \oplus UI$ are left $A$-submodules, as, by Proposition \ref{alsoright}, $A(UI) = (AU)I \subset UI$. We will not need this fact.
\end{rem}

The map $\dim: R \to \Zset$ is a homomorphism of rings, hence $I_A = \ker \dim|_A$ is a two-sided ideal of $A$. Then $J_A = R I_A$ is a left ideal of $R$.
\begin{lemma}\label{irrdim}
One has $\Irr(A) = \{ u \in \Irr(R)\,|\, u - \dim u \in J_A\}$. In particular, the assignment $A \mapsto J_A$ is injective.
\end{lemma}
\begin{proof}
We know that $R$ decomposes into the direct sum of $A$-submodules $A \oplus U$, and correspondingly $J_A = I_A \oplus UI_A$. Let $u \in \Irr(R)$ satisfy $u - \dim u \in J_A$, and assume $u \in \Irr(R) \setminus \Irr(A)$. Then $u \in U$, hence $(-\dim u) + u$ is the unique expression of $u - \dim u$ as sum of an element from $A$ and an element from $U$. As $u - \dim u \in J_A$, then $- \dim u$ belongs to $J_A\cap A = I_A$, which forces $\dim u = 0$, a contradiction.
\end{proof}

\begin{thm}\label{lietype}
If $R$ is Noetherian then $R$ is of Lie type.
\end{thm}
\begin{proof}
Let
$$A_1 < A_2 < \dots < A_n < \dots$$
be an infinite ascending sequence of proper inclusions of sub-representation rings of $R$. Then
$$J_{A_1} \subset J_{A_2} \subset \dots \subset J_{A_n} \subset \dots$$
is an infinite ascending sequence of proper inclusions of right ideals of $R$. However, $R$ is right Noetherian, and we get a contradiction.
\end{proof}

\begin{thm}\label{generatingrep}
If $R$ is of Lie type then $R$ has a generating representation. Equivalently, every compact quantum group of Lie type is a compact matrix quantum group.
\end{thm}
\begin{proof}
Set $Y_0 = \emptyset$, and define inductively $A_i = \langle Y_i \rangle$, $Y_{i+1} = Y_i \cup \{u_{i+1}\}$, where $u_{i+1} \in \Irr(R) \setminus \Irr(A_i)$.
Then
$$A_0 < A_1 < A_2 < \dots < A_n < \dots$$
does not stabilize, hence there must exist $N$ such that $\Irr(R) \setminus \Irr(A_N) = \emptyset$. This forces $R = \langle u_1, \dots, u_N\rangle$.
\end{proof}

\begin{rem}
Assume we have a ring homomorphism $d: R \to \Qset$ which takes non-zero values on irreducible representations. Then $I^d_A = \ker d|_A$ is a two-sided ideal of $A$, and $J^d_A = R I^d_A$ is a left ideal of $R$. The proof of Lemma \ref{irrdim} may be adapted to prove that the assignment $A \mapsto J^d_A$ is injective. Indeed,
$$\Irr(A) = \{ u \in \Irr(R)\,|\, s u - r \in J_A \mbox{ for some non-zero } r, s \in \Zset \mbox{ satisfying } d(u) = r/s\}.$$
A similar argument applies when $K$ is a number field, $\OK$ its ring of algebraic integers, $R_K = \OK \otimes_\Zset R$, $A_K = \OK \otimes_\Zset A$, and we are given a ring homomorphism $d: R \to K$ which is non-zero on irreducible representations.

This may be applied towards hypergroups possessing a dimension function taking values in a number field $K$ and not necessarily associated to compact quantum groups. Indeed, if the fusion ring is Noetherian, then it stays Noetherian after we tensor it by the finitely generated $\Zset$-algebra $\OK$. Then adapting Lemma \ref{irrdim} and Theorem \ref{lietype} shows that the hypergroup satisfies the ascending chain condition on sub-hypergroups, and contains a generating representation. This should be compared with the well known result by Etingof, Nikshych and Ostrik who proved that
any complex-valued homomorphism of the Grothendieck ring of a fusion category (with finitely many inequivalent irreducible representations) takes values in
${\mathbb Q}(\zeta)$, with $\zeta$ some  root of unity. In particular,
the Jones index of a subfactor with finite depth is a cyclotomic integer
 \cite[Theorem 8.51]{ENO}.
\end{rem}

Declare a property (*) to be {\em hereditary} whenever
$$R \mbox{ has property (*)}, A < R \Longrightarrow A \mbox{ has property (*)}.$$
Then being of Lie type is trivially a hereditary property, and Theorem \ref{noetherian} shows that Noetherianity is also hereditary. In the commutative case, being finitely generated implies Noetherianity, and a result of Hashimoto shows that being finitely generated is also a hereditary property
\cite{hashimoto}, cf. Subsection 6.1.

Hereditary properties of Grothendieck rings of representations of a compact quantum group $G$ are inherited by all quotients of $G$. Therefore all quotients of a compact quantum group of Lie type are still of Lie type, and the same holds for the property of having a Noetherian representation ring, whereas this certainly fails for the property of being a compact matrix quantum group, already in the cocommutative case: indeed, not every subgroup of a finitely generated group is finitely generated.

\begin{rem}
In the classical case, if $G$ is a compact  group, and $R$ is its representation ring, then $G$ is a Lie group if and only if $R$ is Noetherian. Indeed, if it is not a Lie group, then it has an infinite strictly increasing sequence of quotients, which is equivalent to $R$ having an infinite strictly increasing sequence of sub-representation rings. Vice-versa, if $G$ is a compact connected Lie group, then we may find a finite cover of $G$ isomorphic to a direct product $K \times T$, where $K$ is a simply connected compact Lie group, and $T$ is a torus.
Then it is easy to show that the representation ring of $K \times T$ equals $R = \Zset[u_1, \dots, u_r, \chi_1^{\pm 1}, \dots, \chi_d^{\pm 1}]$, where $r$ is the rank of $K$, $u_1, \dots, u_r$ are the fundamental representations of $K$, and $d$ is the dimension of $T$. As $R$ is Noetherian, then all of its sub-representation rings, including that of $G$, are too.
In the case where $G$ is a general compact Lie group, Segal showed that $R(G)$ is a finitely generated
ring. In particular, being commutative, it is still Noetherian \cite{Segal}.
We conclude that Noetherianity, being of Lie type, possessing a generating representation and being finitely generated, are equivalent requirements in the classical setting.

A compact quantum group with representation ring isomorphic to that of a compact Lie group is of Lie type. In particular, deformations of the classical groups as well as $A_o(F)$ are of Lie type.
\end{rem}

\begin{ex}
If $G$ is a cocommutative quantum group associated to the discrete group $\Gamma$,
$R(G)$ reduces to the group ring ${\mathbb Z}\Gamma$. Correspondingly, the Lie property becomes the requirement that $\Gamma$ is a Noetherian group, i.e., that it satisfies the ascending chain condition on subgroups. Equivalently, every subgroup is finitely generated. The previous results generalize properties known for group rings  (see, e.g., \cite{Rowen}) to representation rings of compact quantum groups. The examples known in the literature recalled  in the Introduction distinguish the various properties.
\end{ex}

In analogy with the fact that the free groups are not Noetherian, we show the following fact.

\begin{thm}\label{$A_u(F)$}The quantum groups $A_u(F)$ are not of Lie type.
\end{thm}
\begin{proof}
For a positive integer $d$, let $A_d$ be the sub-representation ring of $R(A_u(F))$
generated by $\iota$ and $\{\overline{u}^ru^r, r=1,\dots,d\}$. This is clearly an increasing
sequence of sub-representation rings. We show that $A_d$ is strictly increasing.

Banica \cite{Banica} showed that the irreducible representations of $A_u(F))$ are
labeled  by the elements of the  free unital semigroup ${\mathbb N}*{\mathbb N}$ with the following fusion rules.   The semigroup product and the representation tensor product will be denoted by
$xy$ and $x\otimes y$ respectively.
Let $u$ and $\overline{u}$ be the
generators of ${\mathbb N}*{\mathbb N}$.
One has:
$xu\otimes \overline{u}y=xu\overline{u}y+x\otimes y$, $xu\otimes uy=xu^2y$, and similar
relations with the roles of $u$ and $\overline{u}$ exchanged. It follows that
for $p$, $q$, $r\geq1$, the irreducible subrepresentations of $\overline{u}^pu^q\otimes\overline{u}^ru^r$ are of the following form.
$a_1$) For $r\geq q$,
$\overline{u}^pu^{q-j}\overline{u}^{r-j}u^r$
for $j=0,\dots q-1$  and, in addition, $a_{1,1}$) for $r>q$, $\overline{u}^{p+r-q}u^r$;  $a_{1,2,1}$) $r=q$, $p>r$,
$\overline{u}^{p-j}u^{r-j}$, $j=0,\dots,r-1$, $\overline{u}^{p-r}$
$a_{1,2,2}$) for $r=q$, $p\leq r$, $\overline{u}^{p-j}u^{r-j}$, $j=0,\dots,p-1$
and in addition $a_{1,2,2,1}$) for $r=q$, $p<r$, $u^{r-p}$; $a_{1,2,2,2}$) $r=q$, $p=r$, $\overline{u}u$; $a_2$) $r<q$, $\overline{u}^pu^{q-j}\overline{u}^{r-j}u^r$, $j=0,\dots,r-1$,
$\overline{u}^pu^q$.

For a word of the form $\overline{u}^{p_1}u^{q_1}\dots \overline{u}^{p_k}u^{q_k}$,
with $p_i$, $q_i\geq1$, we refer to $k$ as its length.
An inductive argument on  $n$ shows that  the irreducible subrepresentations
of a tensor product
$\overline{u}^{r_1}u^{r_1}\otimes \dots \otimes \overline{u}^{r_n}u^{r_n}$, $1\leq r_j\leq d$,  if not trivial, are words
$\overline{u}^{p_1}u^{q_1}\dots \overline{u}^{p_k}u^{q_k}$
of length $1\leq k\leq n$ satisfying $\sum_j p_j=\sum _jq_j$. In particular, case $a_{1,2,2,1})$ does not arise for $k=1$. It follows that   $\{\overline{u}^ru^r, 1\leq r\leq d\}$ are all the words of length $1$ obtained in this way. In particular, $\overline{u}^{d+1}u^{d+1}\in A_{d+1}- A_d$.
\end{proof}

\begin{rem}\label{nosubgroups}
Both in the commutative and in the cocommutative case, the property of being of  Lie type is preserved by passing to quantum subgroups. This fact does not hold for general compact quantum groups, even if
the representation ring of the larger group is isomorphic to that of a compact Lie group (hence commutative and finitely generated).
For example, the   quantum group $G=A_o(n)$ of Wang admits, as a subgroup, the cocommutative quantum group associated to the free product $\Cyc_2*\dots*\Cyc_2$ of $n$ copies of the cyclic group of order $2$, see \cite{Wang_free}. This group  is not Noetherian for $n\geq 3$
since it contains the free group ${\mathbb F}_2$.
\end{rem}

\subsection{Applications to finiteness of $G^\circ\backslash G$}\

\medskip

\noindent Theorem \ref{connectedcomponent}  leaves us with the problem of deciding under what conditions the torsion subcategory  $\Rep(G)^t$
associated to a compact quantum group $G$ is tensorial, finite and normal.  A relevant part of the problem is that of ensuring tensoriality and finiteness of $\Rep(G)^t$.
As observed in the introduction, the cocommutative examples lead to consider the case where $\Rep(G)^t$ is commutative as a first class of examples.

\begin{cor}\label{commthentensfinite}
Let $G$ be a compact quantum group of Lie type. Then (all subgroups and) all quotient quantum groups of $G$ are compact matrix quantum groups. In particular, if torsion representations commute (up to equivalence) then $\Rep(G)^t$ is tensorial and finite.
\end{cor}

\begin{proof} The first statement follows from Theorem  \ref{generatingrep} and the fact that the Lie property is hereditary. If $\Rep(G)^t$ is commutative then it is tensorial, or, more precisely, it corresponds to a
quotient quantum group, by  Propositions \ref{categoriesquotients} and   \ref{commutativetorsion}. Hence it admits a generating representation by the previous part, and therefore it must be  finite by commutativity.
\end{proof}

We note that if the Lie property is not assumed, finiteness of the torsion
part fails even assuming   that torsion representations are   central. Indeed, Remeslennikov has constructed   examples of finitely generated discrete groups $\Gamma$ such that the center $Z(\Gamma)$ contains an infinite torsion group with finite exponent \cite{Remeslennikov1, Remeslennikov2}.
Even stronger results have been obtained by Ould Houcine who proved,   among other things, that every countable abelian group is a subgroup of the centre of some finitely presented group \cite{Houcine}. We next combine the main results of the last two sections.

\begin{thm}\label{final} Let $G$ be a compact quantum group of Lie type with commutative and normal torsion subcategory $\Rep(G)^t$.
Then $G$ has torsion degree $\leq1$, $\go$ is a normal quantum subgroup and $\go\backslash G$ is finite. Furthermore,
$\Rep(\go\backslash G)$ identifies with $\Rep(G)^t$.
\end{thm}

In practice, an important class of compact quantum groups $G$ of Lie type with commutative torsion subcategory,   is that for which $R(G)$ is commutative and finitely generated (as a ring). A stronger commutativity requirement involving torsion representations, or other sufficient conditions that are quite easy to verify in specific examples,   ensure normality of $\Rep(G)^t$, by Proposition \ref{sufficientfornormality}.

We conclude this section with a couple of related results for more general rings.
We first notice a relation between hypergroup finite generation and ring finite generation for the Grothendieck ring associated to a general tensor $C^*$-category.
We omit the easy proof.\medskip
\begin{prop}
Let $\T$ be a tensor $C^*$-category with conjugates, subobjects and direct sums. If the Grothendieck ring $R(\T)$ is finitely generated as a ring, then the hypergroup $\hat\T$ of equivalence classes of irreducible objects of $\T$ is finitely generated as well. In other words, $\T$ has a generating object.
\end{prop}

We next recall that  Hashimoto proved the following result on finite generation of certain subrings of commutative rings with methods of algebraic geometry.
\medskip

\begin{thm}[\cite{hashimoto}] Let $Z$ be a Noetherian commutative ring and let
$A\subset R$ be an inclusion of commutative $Z$--algebras such that $R$ is finitely generated over $Z$ and $A$ is a pure.  Then $A$ is finitely generated over $Z$.
\end{thm}
A subring $A$ of a commutative ring $R$ is called a direct summand if there is an $A$--linear map $E:R\to A$ such that $E(a)=a$, $a\in A$. A direct summand subring is pure, i.e., for any $A$--module $M$, the map $m\in M\to m\otimes I\in M\otimes_AR$ is injective.\medskip

This result follows the work of Fogarty \cite{F}, which, in turn, has its roots in the classical problem of finite generation of rings of invariants under a group action. Notably, unlike classical invariant theory, these results on finite generation do not reduce to the graded case.

\section{An example: $\widehat{U_q(\mathfrak{su}_{1,1})}$ for negative values of $q$}
 Our aim in this section is to show that the compact real forms of $U_q(\mathfrak{sl}_2)$, for  $q\in{\mathbb R}$, $q\neq0$, $q\neq\pm1$, are not connected and that the identity component and the quantum component group can be computed with the methods developed in this paper. While the case $q>0$ is widely known,  we shall mostly focus on the case $q<0$.

Recall \cite{Drinfeld1, Drinfeld2, Jimbo} that
the Drinfeld-Jimbo quantum group $U_q(\mathfrak{sl}_2)$,  is the Hopf algebra generated by elements $E$, $F$, $K$ and relations
$$KK^{-1}=K^{-1}K=1,\quad KEK^{-1}= qE,\quad KFK^{-1}=q^{-1}F,$$
$$[E,F]=\frac{K^2-K^{-2}}{q-q^{-1}}$$
with comultiplication $\Delta$, antipode $S$ and counit $\varepsilon$ given by
$$\Delta(K)=K\otimes K,\quad \Delta(E)=E\otimes K^{-1}+K\otimes E,\quad
\Delta(F)=F\otimes K^{-1}+K\otimes F,$$
$$S(K)=K^{-1}, \quad S(E)=-q^{-1} E,\quad S(F)=-qF,
\quad \varepsilon(K)=1,\quad \varepsilon(E)=\varepsilon(F)=0.$$

Representation theory of  $U_{q}(\mathfrak{sl}_2)$ is  well known. There are four inequivalent irreducible representations for each dimension, $\pi_{(w,n)}:=\iota_w\otimes \pi_n$, where $\iota_w$ are $1$-dimensional representations,
$$\iota_w: E\to0,\quad F\to0;\quad K\to w,$$
with $w\in\{{\pm1, \pm i}\}$, and
$$\pi_{n}(E)v_r=[n-r+1]v_{r-1},\quad
\pi_{n}(F)v_r=[r+1]v_{r+1},$$
$$   \pi_{n}(K)v_r= t^{n-2r}v_r,$$
on a linear basis $v_0,\dots, v_n$ where now $t$ is a fixed square root of $q$  and
$$[k]:=\frac{q^k-q^{-k}}{q-q^{-1}}=\frac{q^{2k}-1}{q^{k-1}(q^2-1)}.$$
Tensor product of representations is naturally defined by means of the comultiplication. It is well known that the tensor products of the $\pi_n$'s commute up to canonical invertible intertwiners, the  braiding operators, and that $\pi_n\otimes\pi_m$ decomposes according to  the Clebsch-Gordan fusion rules. Hence $\{\iota_w\}$ is the set of irreducible torsion representations, and they form a finite tensor category corresponding to $\Cyc_4)$. The following simple observation will play a role later.
\begin{rem}\label{torsionpartisnormal}
A direct computation shows that the permutation operators establish equivalence between $\iota_{\pm 1}\otimes\pi_{w, n}$ and $\pi_{w,n}\otimes\iota_{\pm1}$.
\end{rem}
Moreover, $\pi_{w,n}=\iota_w\otimes \pi_n$ is equivalent to $\pi_n\otimes \iota_w$ also for $w=\pm i$ (although not through the permutation operator), therefore all the representations commute up to equivalence.

Recall that $U_q(\mathfrak{su}_{2})$ and $U_q(\mathfrak{su}_{1,1})$ are the Hopf $^*$-algebras derived from $U_q(\mathfrak{sl}_2)$ with the following involutions, respectively \cite{CP, KS},
$$ E^*=F, \quad K^*=K,\quad\quad U_q(\mathfrak{su}_{2}),$$
$$E^*=-F,\quad K^*=K, \quad \quad U_q(\mathfrak{su}_{1,1}).$$
One can derive compact quantum groups in two different ways.
\bigskip

\noindent {\it Case  $U_{q}(\mathfrak{su}_{2})$, $\quad q>0$}\
\medskip

\noindent It is well known that in this case $U_q(\mathfrak{su}_{2})$ has many non-trivial $^*$-representations on finite-dimensional Hilbert spaces, while $U_{q}(\mathfrak{su}_{1,1})$ has none. We need the following results on representation theory of $U_q(\mathfrak{su}_{2})$. For more details, see \cite{Rosso}. The category of finite-dimensional
$^*$-representations of $U_q(\mathfrak{su}_{2})$ on finite-dimensional Hilbert spaces is an embedded tensor $C^*$-category with conjugates. Hence it corresponds to a compact quantum group,  $G=\widehat{U_q(\mathfrak{su}_{2})}$.   Among the $1$-dimensional representations, only $\iota_{\pm 1}$ are $^*$-representations. In particular,  $\iota_{-1}$
generates $\Rep(G)^t$, which is a tensor category isomorphic to the category corresponding to $\Cyc_2$. The following is a complete set of irreducible $^*$-representations  with  positive weights, on an orthonormal basis $(\psi_0,\dots,\psi_n)$,
$$u_{n}(E)\psi_r=
   \sqrt{[n-r+1][r]}\psi_{r-1},\quad u_{ n}(F)\psi_r=
\sqrt{[r+1][n-r]}\psi_{r+1},$$
$$u_{ n}(K)\psi_r= (\sqrt{q})^{n-2r}\psi_r.$$
$G$ admits two irreducible representations for each dimension, $ \iota_{\pm1}\otimes u_n$.
 The $u_n$'s generate a torsion--free full tensor $C^*$-subcategory. These yield, via  Tannaka--Krein duality, the compact quantum group ${\rm SU}_q(2)$ of Woronowicz.
 The process of  restricting attention  to the  representations with positive weights can thus be interpreted as the passage to the identity component. Indeed,
every irreducible of $\Rep(G)$ is determined by a pair constituted by an irreducible of $\Rep({\rm SU}_q(2))$ and one of $\Cyc_2$, hence  we may reconstruct the original quantum group as the product of the identity component and the component group,
$$\widehat{U_q(\mathfrak{su}_{2})}={\rm SU}_q(2)\times\Cyc_2.$$

\noindent From this perspective, the next example  is more interesting.
\bigskip

\noindent{\it Case  $U_{q}(\mathfrak{su}_{1,1})$, $\quad q<0$}
\medskip

\noindent As in the previous example, $\iota_{\pm1}$ are still the $1$--dimensional $^*$-representations. The following fact may be  known. We include a proof as we do not have a reference.

\begin{prop}
For $q<0$, $U_{q}(\mathfrak{su}_{2})$ has no finite-dimensional $^*$-representation on a Hilbert space. $U_{q}(\mathfrak{su}_{1,1})$ admits two inequivalent irreducible Hilbert space $^*$-representations for each dimension $\geq0$. They are given as follows on an orthonormal basis $(\psi_0,\dots,\psi_n)$.

\noindent For $n$ odd,
$$u_{\pm n}(E)\psi_r=
\pm i \sqrt{[n-r+1][r]}\psi_{r-1},\quad u_{\pm n}(F)\psi_r= \pm i
\sqrt{[r+1][n-r]}\psi_{r+1},$$
$$u_{\pm n}(K)\psi_r= \mp (-1)^{\frac{n-1}{2}-r} (\sqrt{|q|})^{n-2r}\psi_r.$$
For $n$  even,
$$u_{\pm n}(E)\psi_r=
\pm(-1)^r \sqrt{-[n-r+1][r]}\psi_{r-1},\quad u_{\pm n}(F)\psi_r= \pm
(-1)^{r}
\sqrt{-[r+1][n-r]}\psi_{r+1},$$
$$u_{\pm n}(K)\psi_r= \pm (-1)^{\frac{n}{2}-r}(\sqrt{|q|})^{n-2r}\psi_r.$$
\end{prop}
\begin{proof}
We may assume $n>0$. Note that the sign of $[k]$ depends on the parity of $k$ and that $t$ is pure imaginary. Moreover, for
$r<n$,
$$\pi_{(w,n)}(EF)v_r=w^2\frac{(q^{2(r+1)}-1)(q^{2(n-r)}-1)}{q^{n-1}(q^2-1)^2}v_r.$$
Let us equip the space $V_{w,n}$ of $\pi_{w,n}$ with an arbitrary Hilbert space structure. By polar decomposition of invertible operators on Hilbert spaces,
$\pi_{(w,n)}$
is equivalent to a $^*$-representation of either $U_q(\mathfrak{su}_{2})$ or
$U_q(\mathfrak{su}_{1,1})$ if and only if there is a positive invertible operator $T$ on $V_{w,n}$ such that $X\to T\pi_{w,n}(X)T^{-1}$ is a $^*$-representation of the corresponding $^*$-algebra. Let us assume that this is the case. Since $K$ is self-adjoint in both algebras,  so is $T\pi_{w,n}(K)T^{-1}$, therefore $\pi_{w,n}(K)$ has real eigenvalues. On the other hand, if $n$ is odd, $t^{n-2r}$ is pure imaginary. Therefore $w=\pm i$. By the above formula, $\pi_{w,n}(EF)$ has negative eigenvalues, hence $T\pi_{w,n}(EF)T^{-1}$ is a negative operator. This is in agreement with $EF=-EE^*$ in $U_q(\mathfrak{su}_{1,1})$ but in contrast with $U_q(\mathfrak{su}_{2})$ where instead $EF=EE^*$.
If $n$ is even, $t^{n-2r}$ is now real, hence $w=\pm 1$, so $\pi_{w,n}(EF)$ has  negative eigenvalues, again in agreement with $U_q(\mathfrak{su}_{1,1})$ and in contrast with $U_q(\mathfrak{su}_{2})$. In particular, $U_q(\mathfrak{su}_{2})$ has no finite-dimensional $^*$-representations.

We  now  show that both $\pi_{(\pm i,\text{odd})}$ and $\pi_{(\pm1,\text{even})}$ are indeed equivalent to $^*$-representations of $U_{q}(\mathfrak{su}_{1,1})$ on a Hilbert space. We introduce an inner product
in the representation space making $\{v_r\}$ into an orthonormal basis, $\{\psi_r\}$.
Since $E^*=-F$ and the image of $K$ is diagonal,
it suffices to find an invertible diagonal matrix
 $T=\text{diag}(t_1,\dots, t_{n+1})$ with complex entries such that
$(T\pi(E)T^{-1})^*=-T\pi(F)T^{-1}$, where $\pi$ is either $\pi_{(\pm
i,\text{odd})}$ or $\pi_{(\pm1,\text{even})}$.
We are thus reduced to solve
$\pi(E)^*T^*T=-T^*T\pi(F)$. Explicitly, we need
$\overline{w}[n-r+1]|t_r|^2=-w[r]|t_{r+1}|^2$,
where $w$ takes the allowed values according to the parity of $n$. Specifically, for $n$ odd, $\overline{w}=-w$ and
$[n-r+1]$ and $[r]$ have the same sign, hence we may solve inductively and find positive
entries $t_1=1$, $t_2=\sqrt{[n][1]^{-1}}$,
$t_3=\sqrt{[n][n-1]([1][2])^{-1}},\dots$,
giving the desired $^*$-representation $u_{\pm n}$.
 If $n$ is even,
$\overline{w}=w$ and
$[n-r+1]$ and $[r]$ have now opposite sign,
and we may still find positive entries
$t_1=1$, $t_2=\sqrt{-[n][1]^{-1}}$,
$t_3=\sqrt{(-[n])(-[n-1])([1][2])^{-1}},\dots$,
giving again the stated $^*$-representation $u_{\pm n}$.
 \end{proof}

\begin{rem} The main difference with the example of the previous subsection is that for $n$ odd,
$u_{\pm n}$ is (not unitarily) equivalent  to $\iota_{\pm i} \otimes \pi_n$.
However, neither $\iota_{\pm i} $ nor $\pi_n$ are equivalent to $^*$-representations.
This phenomenon does not occur in the even case, as
 $u_{\pm \text{n}}$ is equivalent to $\iota_{\pm 1}\otimes \pi_{\text{n}}$.
\end{rem}

\noindent {\it Fusion rules of irreducible representations of $U_{q}(\mathfrak{su}_{1,1})$.}\
\medskip

\noindent  We write down the fusion rules of irreducible $^*$-representations on Hilbert spaces.  They may be easily
derived from the Clebsch--Gordan rules for  the representations
$\pi_n$ of $U_q(\mathfrak{sl}_2)$ and the fact that as representations, for a suitable   $w$,
$u_{\pm n}\simeq \iota_w\otimes\pi_n\simeq \pi_n\otimes\iota_w$.\medskip

In the following, the sums  involve terms with either even or odd indices and we assume $m$, $n\geq0$. We omit relations that can be obtained commuting the factors.
\medskip

\noindent For $m$, $n$ odd,
$$u_{n} u_{m}\simeq u_{-|n-m|}+\dots+ u_{-(n+m)}\simeq u_{-n}
 u_{-m},$$
$$u_{n} u_{-m}\simeq u_{|n-m|}+\dots+ u_{n+m}\simeq u_{-n}
 u_{m},$$
$m$, $n$ even,
$$u_{n}
 u_{m}\simeq u_{|n-m|}+\dots+ u_{n+m}\simeq u_{-n}
  u_{-m},$$
$$u_{n}
 u_{-m}\simeq u_{-|n-m|}+\dots+ u_{-(n+m)}\simeq u_{-n}
  u_{m},$$
$n$ odd, $m$ even,
$$u_{n}
 u_{m}  \simeq u_{ |n-m|}+\dots+ u_{n+m}\simeq u_{-n}
  u_{-m},$$
$$u_{-n}
 u_{m} \simeq u_{-|n-m|}+\dots+ u_{-(n+m)}\simeq u_{n} u_{-m}.$$
The fourth and last line in particular imply
$$\iota_{-1}
 u_n\simeq u_{-n}\simeq u_n
  \iota_{-1}.$$

  \noindent {\it The  associated compact quantum group
  $\widehat{U_q(\mathfrak{su}_{1,1})}$.}\
\medskip

Finite direct sums of irreducible Hilbert space $^*$-representations of $U_q(\mathfrak{su}_{1,1})$ form an embedded tensor $C^*$-category with subobjects and direct sums.
We claim that this category has conjugates. It suffices to show that every irreducible has a conjugate.

If $\overline{u}$ is a conjugate of $u$, then   $\iota<\overline{u}u$. The fusion rules show that for $n$ odd, $\iota<u_{-n}u_n$ and for $n$ even $\iota<u_n^2$, both inclusions have   multiplicity $1$. We need to show that indeed   $u_{-n}$ or $u_n$ is a  conjugate
of $u_n$ if $n$ is odd or even respectively.
 We start with the case $n=\pm1$, the other cases will easily follow from Theorem \ref{conjugatesexample}.

\begin{prop}\label{generators} Set $u:=u_1$ and $\overline{u}:=u_{-1}$. Up to scalars, with respect to the
 orthonormal basis $\{\psi_0,\psi_1\}$ of the Hilbert space of $u$ and $\overline{u}$,
 the arrows $R\in(\iota, \overline{u}u)$ and $\overline{R}\in(\iota, u\overline{u})$ are given by
  $$R=\psi_0\otimes\psi_1-|q|\psi_1\otimes\psi_0=\overline{R}.$$
  In particular, $\overline{u}$ is a conjugate of $u$.
\end{prop}

\begin{proof} The proof proceeds by straightforward computations.
 We write $\overline{R}(1):=\sum_0^1 a_{i,j}\psi_i\otimes\psi_j$, and the intertwining relations $\overline{R}\iota(X)=u\otimes\overline{u}(\Delta(X))\overline{R}$. For $X=E$, the left hand side annihilates, and the relation gives $a_{1,1}=0$,
 $a_{1,0}=-a_{0,1}|q|$. For $X=F$, one obtains  in addition $a_{0,0}=0$.  The relation for $X=K$ now follows automatically. To determine $R$ we may use
  $u_{-1}(X)=-u(X)$ for $X=E$, $F$, $K$.
\end{proof}

\begin{rem} Although the formula for $R$ and $\overline{R}$ reminds of the canonical generator of the representation category of ${\rm SU}_{|q|}(2)$, here $u$ is not self-conjugate, as
$u^2\not >\iota$.
\end{rem}

\begin{thm}\label{conjugatesexample} The category of finite direct sums of irreducible $^*$-representations of $U_q(\mathfrak{su}_{1,1})$ on Hilbert spaces is generated, as a tensor $C^*$-category with subobjects and direct sums, by the objects $\iota_{-1}$, $u=u_1$ and a pair of arrows
 $R\in(\iota, \overline{u} u)$,
$\overline{R}\in(\iota, u \overline{u})$ solving the conjugate equations, where $\overline{u}:=\iota_{-1}u$. In particular, this category has conjugates.
   \end{thm}
\begin{proof}
Let ${\mathcal T}$ denote the smallest tensor $^*$-subcategory with subobjects and direct sums  containing objects $u$, $\iota_{-1}$ and arrows $R$, $\overline{R}$. Since $u$ and $\iota_{-1}$ have conjugates in ${\mathcal T}$, so does ${\mathcal T}$. The space of arrows in
${\mathcal T}$ between two objects    will be denoted $(u,v)_{{\mathcal T}}$.
We need to show that ${\mathcal T}$ contains a complete set of irreducible representations.
Since $u_{-n}\simeq \iota_{-1} u_n$, it suffices to show that $u_n\in{\mathcal T}$ for $n\geq 2$.
Since
$ \overline{u} u=\iota+ u_2$ in $\Rep(G)$,  and  $0\neq R\in(\iota, \overline{u}
 u)_{\mathcal T}$,   $u_2$ is the subobject of $\overline{u}
  u$ defined by  the projection orthogonal to the range of ${R}$. Hence $u_2\in{\mathcal T}$ and the decomposition holds in ${\mathcal T}$. This implies  $(u_2, \overline{u}
   u)_{\mathcal T}\neq0$.
On the other hand, the conjugate equations of $u$ induce a linear isomorphism, Frobenius reciprocity,   $T\in (u_2, \overline{u}
 u)_{\mathcal T}\to\overline{R}^*\otimes 1_u\circ 1_u\otimes T\in (u u_2, u)_{\mathcal T}$. Hence $(u u_2, u)_{\mathcal T}\neq0$ as well. But $u u_2=u+ u_3$, hence, as before, $u_3\in{\mathcal T}$.
Iteratively, we obtain: If $n$ is odd and $u_n\in{\mathcal T}$, since
$uu_{n-1}=u_{n-2}+ u_{n}$ in ${\mathcal T}$, then $0\neq (u_n, u u_{n-1})_{{\mathcal T}}\simeq
(\overline{u}u_n, u_{n-1})_{{\mathcal T}}$. But $\overline{u}u_n=u_{n-1}+u_{n+1}$ in $\Rep(G)$,
hence $u_{n+1}$, as well as the decomposition, are in ${\mathcal T}$. If $n$ is even and $u_n\in{\mathcal T}$,
$\overline{u}u_{n-1}=u_{n-2}+u_n$, which similarly implies
$0\neq (u_n, \overline{u}u_{n-1})_{\mathcal T}\simeq (uu_n, u_{n-1})_{\mathcal T}$. Now $uu_n= u_{n-1}+u_{n+1}$ hence $u_{n+1}\in{\mathcal T}$.
\end{proof}

We can now apply Tannaka--Krein--Woronowicz duality.

\begin{cor}
There is  a compact quantum group, $\widehat{U_q(\mathfrak{su}_{1,1})}$,
with this representation category.
\end{cor}

 \begin{rem}
The proof  of Theorem \ref{conjugatesexample}  shows that every irreducible $u_{\pm n}$
admits a polynomial expression in $u$ and $\iota_{-1}$ in the commutative ring $R(\widehat{U_q(\mathfrak{su}_{1,1})}$.
In particular, $R(\widehat{U_q(\mathfrak{su}_{1,1})}$ is Noetherian.
 \end{rem}

  \noindent {\it The  identity component and the component group of
  $\widehat{U_q(\mathfrak{su}_{1,1})}$.}\
\medskip

We next identify the identity component and the quantum component group.
We shall need the well known variant of $U_q(\mathfrak{sl}_2)$, that we shall denote by
$U_q(\mathfrak{sl}_2)^{\widetilde{}}$. This is the Hopf algebra generated by $E$, $F$, $K$
and relations

$$KK^{-1}=K^{-1}K=1,\quad KEK^{-1}= q^2E,\quad KFK^{-1}=q^{-2}F,$$
$$[E,F]=\frac{K-K^{-1}}{q-q^{-1}}$$
with comultiplication $\Delta$, antipode $S$ and counit $\varepsilon$
 given by
$$\Delta(K)=K\otimes K,\quad \Delta(E)=E\otimes I+K\otimes E,\quad
\Delta(F)=F\otimes K^{-1}+I\otimes F,$$
$$S(K)=K^{-1}, \quad S(E)=-K^{-1} E,\quad S(F)=-FK,
\quad \varepsilon(K)=1,\quad \varepsilon(E)=\varepsilon(F)=0.$$
This Hopf algebra has two $^*$-involutions making it into a Hopf $^*$-algebra,
 $$ E^*=FK, \quad F^*=K^{-1}E,\quad K^*=K,\quad\quad U_q(\mathfrak{su}_2)^{\widetilde{}},$$
$$E^*=-FK,\quad F^*=-K^{-1}E\quad K^*=K, \quad \quad U_q(\mathfrak{su}_{1,1})^{\widetilde{}}.$$
There is a canonical isomorphism of Hopf $^*$-algebras
$$U_{-q}(\mathfrak{su}_{1,1})^{\widetilde{}}\to U_{q}(\mathfrak{su}_{1,1})^{\widetilde{}},$$
$$E\to E,\quad F\to -F, \quad K\to K.$$

\begin{rem} For $q<0$,  $U_q(\mathfrak{su}_{1,1})$ and $U_{-q}(\mathfrak{su}_2)$ are
not isomorphic as Hopf $^*$-algebras, as
if $n$ is odd,
$  u_n^2\not >\iota$
hence, unlike the irreducible representations of $U_{-q}(\mathfrak{su}_{2})$,  it is not self-conjugate.
\end{rem}

\begin{thm} Set $G=\widehat{U_q(\mathfrak{su}_{1,1})}$,  $q<0$. Then $\go$ is a normal quantum subgroup isomorphic to ${\rm SU}_{|q|}(2)$ and $\go\backslash G\simeq\Cyc_2$.
\end{thm}
\begin{proof}
$\Rep(G)^t$ is tensorial and generated by $\iota_{-1}$. Moreover,  by Remark \ref{torsionpartisnormal} and Proposition \ref{sufficientfornormality} b) it is strongly normal.  Hence
by Theorem  \ref{connectedcomponent}, $\go$ is a normal quantum subgroup
and $\go\backslash G\simeq\Cyc_2$.   To complete the proof,
we need to show that $\go={\rm SU}_{|q|}(2)$, or, in other words, that $G$ admits   ${\rm SU}_{|q|}(2)$ as a quantum subgroup and that every connected quantum subgroup of $G$ is a subgroup of  ${\rm SU}_{|q|}(2)$.
It is well known that $U_q(\mathfrak{su}_{1,1})$ naturally contains   $U_q(\mathfrak{su}_{1,1})^{\widetilde{}}$ as the Hopf $^*$-subalgebra generated by $E'=KE$, $F'=FK^{-1}$ and $K'=K^2$.  We can thus consider the tensor $^*$-category   with subobjects and direct sums
generated by the restrictions of the Hilbert space $^*$-representations of
$U_q(\mathfrak{su}_{1,1})$ to $U_q(\mathfrak{su}_{1,1})^{\widetilde{}}$. We obtain all the Hilbert space $^*$-representations of $U_q(\mathfrak{su}_{1,1})^{\widetilde{}}$ with positive weights on $K'$. This is an embedded tensor $C^*$-category with conjugates, hence it corresponds to a compact quantum group $G'$.
The restriction functor gives $G'$ as a quantum subgroup of $G$. On the other hand,
 $U_q(\mathfrak{su}_{1,1})^{\widetilde{}}$ is canonically isomorphic to  $U_{|q|}(\mathfrak{su}_{2})^{\widetilde{}}$. By \cite{Rosso},  $G'$ is naturally isomorphic to ${\rm SU}_{|q|}(2)$. If $C$ is a connected quantum subgroup of $G$ then $\iota_{-1}$ restricts to the trivial representation of $C$. This holds in particular for  $C={\rm SU}_{|q|}(2)$. Taking into account the explicit form of the generators of $\Rep(G)$ given in Prop.
 \ref{generators}, we see that the restriction functor $\Rep(G)\to\Rep({\rm SU}_{|q|}(2))$ takes $R$ and $\overline{R}$ to the canonical generator of $\Rep({\rm SU}_{|q|}(2))$, and this holds also if ${\rm SU}_{|q|}(2)$ is replaced by $C$. Hence $C$ is a quantum subgroup of ${\rm SU}_{|q|}(2)$.
\end{proof}

\begin{rem}
The   identification of $\go$ in the proof emphasizes the role of the inclusion
 $$U_{|q|}(\mathfrak{su}_{2})^{\widetilde{}}\subset
U_q(\mathfrak{su}_{1,1}).$$ Alternatively, we may apply part d) of Theorem
\ref{connectedcomponent},  together with Prop. \ref{generators} and an explicit computation of the torsion vectors of tensor powers of $u$.
\end{rem}

\begin{rem} Since $\Rep(\widehat{U_q(\mathfrak{su}_{1,1})})$ is generated by $u$ and $\overline{u}$, which are irreducible and free, unlike the example of the previous subsection, $\widehat{U_q(\mathfrak{su}_{1,1})}$ is not
the product of the identity component and the component group.
\end{rem}

\appendix
\section{Subquotients of quantum groups}\label{subquotients}\

\noindent
If $L$ is a quotient quantum group of $G$ defined by $\varphi:\Q_L\to\Q_G$ and $K$ is a (not necessarily normal) quantum subgroup of $G$, defined by $\pi:Q_G\to Q_K$, then $\pi\circ \varphi(\Q_L)$ is a Hopf $^*$-subalgebra of $\Q_K$, and it is not hard to verify that its norm completion in $Q_K$ is a compact quantum group, denoted $[K]$, which, by construction, is a quantum subgroup of $L$ and a quotient quantum group of $K$. We shall refer to $[K]$ as the {\em  image} of $K$ in $L$.

In the Tannakian formalism, images of subgroups are described as follows. Consider the full subcategory $\T\subset \Rep(K)$ whose objects are subobjects of the restricted objects of $L$.
Since the restriction functor  $\Rep(G)\to\Rep(K)$ yields a tensor $^*$-functor, $\T$ is a tensor $^*$-subcategory with conjugates, subobjects and direct sums. Hence it corresponds to a compact quantum group, which coincides with $[K]$.
We thus obtain a commutative diagram
\[\xymatrix{
\Rep(L) \ar[d]_{\text{full}} \ar[r]
& \Rep([K])\ar[d]^{\text{full}} \\
\Rep(G)\ar[r] &\Rep(K) }\]
We may consider an alternative   notion of a subquotient, namely that of a quantum subgroup $F$ of a quotient quantum group $L$, and one immediately sees that already in the cocommutative case the notion of an image is in general strictly stronger than that of a subgroup of a quotient.
Indeed,  we may take $G=C^*(\Gamma)$ and    $L=C^*(\Lambda)$,
$F=C^*(\Omega\backslash\Lambda)$ with $\Omega\subset\Lambda\subset\Gamma$ an inclusion of groups such that
$\Omega$ is normal in $\Lambda$.
On the other hand, the image of  a quantum subgroup $K$ of $G$ in $L$ has the same form as $F$,  but  $\Omega$ is now required to be  normal in $\Gamma$.  Moreover,  the assumption  that $L$ is the quotient  $N\backslash G$  by a normal quantum subgroup of $G$ would not remedy the situation.

However, $G={\rm SU}_q(2)$ yields an important example where the two notions of a subquotient coincide. Quantum subgroups of ${\rm SU}_q(2)$ and ${\rm SO}_q(3)$ have been classified by Podles \cite{Podles}. He described all subgroups of ${\rm SO}_q(3)$ in terms of those of ${\rm SU}_q(2)$ showing indeed that they are precisely the quotients of intermediate subgroups $\Cyc_2\subset J\subset {\rm SU}_q(2)$ by the central subgroup $\Cyc_2\subset{\rm SU}_q(2)$.

As far as we know, a positive solution to the subquotient problem is missing in the general case. We shall restrict our considerations to the case where $L$ itself is the quotient $N\backslash G$ by a normal quantum subgroup $N$ of $G$.

Consider an intermediate  quantum subgroup $N\subset J\subset G$. Then  $N$ is normal in $J$ and moreover  it  is easy to see that  the image $[J]$ in $N\backslash G$ is the quotient quantum group $N\backslash J$.
 On the other hand, any image  in $N\backslash G$ is of this form. Indeed,   if $K$ is a generic quantum subgroup of $G$,
 we may consider the quantum subgroup `generated' by $N$ and $K$, denoted
 $NK$, defined by the embedded tensor category
 $$(u\upharpoonright_{NK}, v\upharpoonright_{NK}):=(u\upharpoonright_{N}, v\upharpoonright_{N})\cap(u\upharpoonright_{K}, v\upharpoonright_{K}).$$
Clearly, $NK$ contains $N$ and $K$ as quantum subgroups, and it is the smallest with this property.
In particular, $NK$ is intermediate between $N$ and $G$.
On the other hand,  $K$ and $NK$ have the same image in
$N\backslash G$. This is due to the fact that  representations of $N\backslash G$ act trivially on $N$,
therefore, by  the previous formula,  the  arrows between the restrictions of
two representations $u$, $v\in\Rep(N\backslash G)$, regarded as representations of $G$,  to $NK$ and $K$ are the same.

Let now $F$ be a quantum subgroup of $N\backslash G$.  We identify the Hopf $*$-algebra of $N\backslash G$ with a Hopf $^*$-subalgebra
of that of $G$. The quotient space $F\backslash(N\backslash G)$, being  an embedded action
of $N\backslash G$, thus becomes an embedded   action of $G$.   For any  irreducible representation $u$ of $\Rep(N\backslash G)$, the spectral space $F_u$ for the $G$--action coincides with the spectral space for the corresponding $N\backslash G$--action,
while  for $u\in\Rep(N\backslash G)^\perp$, $F_u=0$.
Note that if $F=N\backslash J$ is
the image of an intermediate subgroup $J$, this construction gives the spectral spaces $J_u$ for the
$G$--action on $J\backslash G$, and therefore determines $J$ by Theorem \ref{caratt}. In general, it may  fail to correspond to an intermediate quantum subgroup. Indeed, by the same theorem, we may infer  that the extended
$F_u$'s are the spectral spaces associated to a quantum   subgroup $J$ of $G$  if and only if they satisfy relation $(3.1)$ for every irreducible
representation $v\in\Rep(N\backslash G)$ and  $u\in\Rep(G)$.
Note that if this holds, the corresponding quantum subgroup $J$ of $G$ is necessarily intermediate between $N$ and $G$, as $J_u=F_u\subset N_u=H_u$ for all irreducible $u\in\Rep(G)$, and $[J]=F$.
But, being $F$ a subgroup of $N\backslash G$,  relation $(3.1)$ already holds for $u\in\Rep(N\backslash G)$, hence it needs to be verified only for $u\in\Rep(N\backslash G)^\perp$. We next discuss some cases of interest for this paper.

The previous arguments summarize  in the following result. The proof  can be completed with the same arguments as
in Proposition \ref{sufficientfornormality}.

\begin{thm}
Let $N$ be   a strongly normal   quantum subgroup of $G$. Then every quantum subgroup of $N\backslash G$ is the image of a quantum subgroup of $G$.
\end{thm}

\begin{rem}
It follows that if $G^\circ$ is strongly normal then $G^\circ\backslash G$ is totally disconnected, see also Remark \ref{totdisc3}.
\end{rem}

Let us now pass from the categorical to the algebraic description of the quantum groups. Given a quantum subgroup $F$ of $N\backslash G$ we denote by $q:{\mathcal Q}_{N\backslash G}\to{\mathcal Q}_F$ the defining epimorphism. We consider the special case where $N$ is a central quantum subgroup. The Hopf subalgebra ${\mathcal Q}_{N\backslash G}$ of ${\mathcal Q}_G$
is then globally invariant under
the maps $x\to \overline{u}_{\psi,\varphi}x u_{\psi',\varphi'}$, $u\in\Rep(G)$, by Proposition \ref{centrality}.

\begin{thm}
Let $N$ be   a central quantum subgroup of $G$. Then
a quantum subgroup $F$ of $N\backslash G$  is the image of a quantum subgroup of $G$ if and only if for every pair of irreducible representations $v\in\Rep(N\backslash G)$, $u\in\Rep(N\backslash G)^\perp$,
  $$q(\overline{u}_{\psi,\varphi}v_{\xi, f}u_{\psi',\varphi'})=\langle\xi,f\rangle q(\overline{u}_{\psi,\varphi}u_{\psi',\varphi'}),\eqno(8.1)$$
  for $f\in F_v$ and arbitrary vectors $\psi,\varphi,\psi',\varphi',\xi$.
\end{thm}

\begin{proof} If $F=[K]$ with $K$ defined by the epimorphism $\pi:{\mathcal Q}_G\to{\mathcal Q}_K$, then any matrix coefficient of  $v\in\Rep(N\backslash G)$, satisfies $q(v_{\psi,\varphi})=\pi(v_{\psi,\varphi})$.
Since $\pi$ is multiplicative and $q(v_{\xi, f})=\langle\xi, f\rangle$,
the necessity of  relation $(8.1)$ follows.
For the converse,
we need to show that, if $(8.1)$ holds,
there is a quantum subgroup $K$  of $G$ such that the defining epimorphism $\pi:{\mathcal Q}_G\to{\mathcal Q}_K$ restricts to $q$ on ${\mathcal Q}_{N\backslash G}$.
It is perhaps well known  that $\text{ker}(q)={\mathcal I}$, where
${\mathcal I}$ is the ideal  of ${\mathcal Q}_{N\backslash G}$ generated
by elements of the form $v_{\xi, f}-\langle \xi, f\rangle$, where $v\in\Rep(N\backslash G)$ is irreducible, $\xi$ is an arbitrary element of the space of $v$, while $f$ is an arbitrary fixed vector for the restriction of $v$ to $F$, hence an element of $F_v$.
Indeed, on one hand  the inclusion $\text{ker}(q)\supset{\mathcal I}$ is obvious. On the other, it is easy to see that the  ideal ${\mathcal I}$ is
also a coideal, and, since ${\mathcal I}\subset\text{ker}(q)$, it defines an intermediate quantum subgroup $F\subset \tilde{F}\subset N\backslash G$ satisfying $F_v\subset\tilde{F}_v$ for all $v\in\Rep(N\backslash G)$ by construction. Hence $F_v=\tilde{F}_v$, and this implies $F=\tilde{F}$, hence $\text{ker}(q)={\mathcal I}$.
Let us now regard the quotient space $F\backslash(N\backslash G)$ as an embedded action of $G$. As mentioned before, the associated spectral spaces extend the
$F_v$'s  by $F_u=0$ if $u\in\Rep(N\backslash G)^\perp$. By \cite[Theorem 5.1]{PR},
there is a quantum subgroup $K$ of $G$
such that $F_u\subset K_u$ for all
irreducible representations $u\in\Rep(G)$. By construction,  and also
by the form of the spectral spaces of the extended $G$--action, $K$ is defined by the epimorphism
${\mathcal Q}_G\to{\mathcal Q}_K$
whose kernel is the ideal ${\mathcal J}$, now  of ${\mathcal Q}_G$, generated by the same set of generators of ${\mathcal I}$, namely
$v_{\xi, f}-\langle \xi, f\rangle$. We are thus left to show that ${\mathcal J}\cap{\mathcal Q}_{N\backslash G}={\mathcal I}$, or, in other words, that ${\mathcal J}\cap{\mathcal Q}_{N\backslash G}\subset{\mathcal I}$, since the reverse inclusion is obvious. Let us
consider the projection $E:{\mathcal Q}_G\to{\mathcal Q}_{N\backslash G}$
that comes from the splitting of the irreducible objects of $\Rep(G)$ into those
of $\Rep(N\backslash G)$ and $\Rep(N\backslash G)^\perp$, and let $X\in{\mathcal J}\cap{\mathcal Q}_{N\backslash G}$. By Lemma \ref{splitinj}, $E$ is ${\mathcal Q}_{N\backslash G}$--bilinear. Since $X=E(X)$, we may assume that $X$ is a finite sum of elements  of the form $Y:=u'_{\psi', \varphi'}(v_{\xi, f}-\langle \xi, f\rangle)u_{\psi,\varphi}$, where $v\in\Rep(N\backslash G)$, $u$, $u'\in\Rep(N\backslash G)^\perp$,
$f\in F_v$ and the remaining vectors are arbitrary.
If $s$ is a normalized $G$--fixed vector for $u\overline{u}$ then $(u\overline{u})_{s,s}=1$, hence we can write
$$Y=\sum_{i,j}(u'_{\psi',\varphi'}u_{\psi_i,\psi_j})(\overline{u}_{\varphi_i,\varphi_j}(v_{\xi, f}-\langle \xi, f\rangle)u_{\psi,\varphi}),$$
where $\sum_i\psi_i\varphi_i=s$.
Note that the the factors $\overline{u}_{\varphi_i,\varphi_j}(v_{\xi, f}-\langle \xi, f\rangle)u_{\psi,\varphi}$ lie in ${\mathcal Q}_{N\backslash G}$, by centrality of $N$ and Proposition \ref{centrality}. On the other hand, both $u'$ and $u$ restrict to multiples of non-trivial one-dimensional representations of $N$, say $g'$ and $g$,  respectively. If $g'g\neq 1$, $u'u$ contains no subrepresentation which becomes trivial on $N$, hence $u'u\in\Rep(N\backslash G)^\perp$, implying $E(Y)=0$. If, instead, $g'g=1$ then
all the subrepresentations of $u'u$ restrict to a multiple of the trivial representation of $N$, so $u'_{\psi',\varphi'}u_{\psi_i,\psi_j}$, and hence also $Y$, lie in ${\mathcal Q}_{N\backslash G}$, and $q(Y)=0$ by our assumption. This shows that $X\in{\mathcal I}$, and the proof is complete.
\end{proof}

\begin{rem}
In the cocommutative case, as noticed before, every quantum subgroup is central. Resuming the previous notation, with data  described by
$\Omega\subset\Lambda\subset\Gamma$, condition $(8.1)$
then becomes the requirement of normality of $\Omega$ in $\Gamma$.
\end{rem}

 \bigskip

\noindent{\bf Acknowledgments} Part of the results of this paper have been presented at the  conference {\em Noncommutative geometry and quantum groups} held in Oslo in the summer 2012. C.P. would like to thank E. Bedos, M. Landstad, N. Larsen, S. Neshveyev, and L. Tuset for the invitation.

We are  grateful to S. Wang for sending us the paper \cite{Wang_adjoint}, to M. Chiodo, M. Hashimoto, V.N. Remeslennikov,  for correspondence, and to C. Bernardi, J. Bichon,  K. De Commer,  C. De Concini, M. Landstad, R. Meyer, I. Patri, R. Salvati Manni,  A. Van Daele,  for discussions.

L.S.C. acknowledges support by the National Research Fund, Luxembourg,
AFR project 1164566. A.D'A., C.P., and S.R. acknowledge support by AST fundings from ``La Sapienza'' University of Rome.

\vfill

\end{document}